\documentclass[11pt]{article}
\usepackage{amssymb} %for Blackboard bold etc
\usepackage{graphicx} %for including eps graphics
\usepackage{amsthm} %for defining theorems
\usepackage{color}
\usepackage{amsmath}
\usepackage{mathrsfs} 
\usepackage{mathtools} 
\usepackage{dsfont}
\usepackage[parfill]{parskip}
\usepackage[a4paper, total={6in, 8in}]{geometry}
\usepackage{bbm}
\usepackage{hyperref}
\usepackage{cleveref} % Must be last referencing package loaded
\usepackage{enumerate}
\usepackage[normalem]{ulem} % crossout!

\usepackage{caption}
\usepackage[labelformat=simple]{subcaption}

\usepackage{tikz}
\usetikzlibrary{decorations.pathmorphing}

\usepackage{setspace}
\newcommand{\Palm}{\scriptscriptstyle(0,U_0)}

\newcommand{\mathd}{\mathrm{d}}

\newcommand{\R}{\mathbb R}

\newcommand{\E}{\mathbb E}
\renewcommand{\P}{\mathbb P}

\newcommand{\cG}{\mathcal{G}}

\renewcommand{\phi}{\varphi}

\newcommand{\Z}{\mathbb Z}
\newcommand{\N}{\mathbb{N}}
\renewcommand{\P}{\mathbb P}
\newcommand {\T}{\mathbb T} %torus

\newcommand{\X}{\mathcal X}

\newcommand{\0}{\mathbf{0}}
\newcommand{\x}{\mathbf{x}}
\newcommand{\y}{\mathbf{y}}
\newcommand{\z}{\mathbf{z}}
\renewcommand{\v}{\mathbf{v}}
\renewcommand{\k}{\mathbf{k}}

\newcommand{\abs}[1]{\left| #1 \right|}

\newcommand{\1}{\mathbf{1}}

\newcommand{\epsa}{\varepsilon_a}

\newcommand{\epst}{\varepsilon_\theta}

\title{Geometric scale-free random graphs on mobile vertices: broadcast and percolation times}
\author{Peter Gracar$^\dag$\footnote{Corresponding author, E-Mail: p.gracar@leeds.ac.uk}\, and Arne Grauer$^\ddag$\\
{\sl\small$^\dag$University of Leeds, UK}\\
{\sl\small$^\ddag$University of Cologne, Germany}}
\date{}

\usepackage[style=numeric, backend = bibtex, url = false, doi = true, isbn = false, maxnames = 10]{biblatex}
\usepackage{todonotes}

\theoremstyle{plain} %theorems
\newtheorem{theorem}{Theorem}%[section]
\newtheorem{prop}{Proposition}[section]
\newtheorem{lemma}{Lemma}[section]

\newtheorem{assum}{Assumption}%[section]

\crefname{theorem}{theorem}{theorems}
\crefname{prop}{proposition}{propositions}
\crefname{corollary}{corollary}{corollaries}

\theoremstyle{definition} %definitions
\newtheorem{mydef}{Definition}[section]

\crefname{mydef}{definition}{defitnisions}

\theoremstyle{remark} %notes and remarks
\newtheorem{remark}{Remark}[section]

\crefname{remark}{remark}{remarks}

\begin{document}

\maketitle
\begin{abstract}
We study the phenomenon of information propagation on mobile geometric scale-free random graphs, where vertices instantaneously pass on information to all other vertices in the same connected component. The graphs we consider are constructed on a Poisson point process of intensity $\lambda>0$, and the vertices move over time as simple Brownian motions on either $\R^d$ or the $d$-dimensional torus of volume $n$, while edges are randomly drawn depending on the locations of the vertices, as well as their a priori assigned marks. This includes mobile versions of the \emph{age-dependent random connection model} and the \emph{soft Boolean model}. We show that in the ultrasmall regime of these random graphs, information is broadcast to all vertices on a torus of volume $n$ in poly-logarithmic time and that on $\R^d$, the information will reach the infinite component before time $t$ with stretched exponentially high probability, for any $\lambda>0$.
\end{abstract}

\newpage
\tableofcontents

\section{Introduction}

In recent years scale-free geometric random graphs have been used extensively to study real world networks, such as social or telecommunication networks. A broad collection of literature has emerged, focusing on different connection rules that lead to such networks, such as for example the spatial graph model with local influence regions \cite{Aiello2008}, scale-free percolation \cite{Deijfen2013,Deprez2019} and more recent work on spatial preferential attachment type models \cite{Gracar2019}. Further work focused on various topological properties of the emerging networks, such as how graph distances scale with respect to the graph size \cite{Dereich2012} (for non-spatial models) or the euclidean distance between two typical vertices (for spatial models), see \cite{Gracar2022a}. A related phenomenon of interest is the percolation probability in such networks, studied for example in \cite{Baumler2023}.

As the body of literature on these graphs grew, a parallel interest in the study of dynamics related them started developing. Many authors have studied the contact process and how it behaves on and spatial and non-spatial versions of such graphs, starting with Chatterjee and Durrett \cite{Chatterjee2009} and more recently Mourrat and Valesin \cite{Mourrat2016}, Linker \emph{et al} \cite{Linker2021} and others \cite{Schapira2021,Gracar2022b}. Further dynamics have also been studied in non-spatial networks, where the contact process is evolving on graphs that are themselves evolving, such as  edges of the graph being randomly resampled, see for example the recent work by Jacob \emph{et al} \cite{Jacob2022}.

In this paper, we approach the question of dynamics from a different direction. Drawing inspiration from the work done by Peres \emph{et al} \cite{Peres2011}, where they considered the Boolean model with mobile vertices, we study properties of geometric scale-free random graphs where the underlying point process is allowed to evolve over time. More precisely, we study scale-free random graphs where edges emerge and disappear due to the motion of the vertices, which individually move as Brownian motions on euclidean space. Particle systems of this kind have been studied with the help of multi-scale analysis techniques on lattices by Kesten and Sidoravicius \cite{Kesten2005}, uniformly elliptic lattices by Gracar and Stauffer \cite{Gracar2019b}, fractal lattices by Drewitz \cite{Drewitz2023} and on euclidean space by Stauffer \cite{Stauffer2015}, with the primary focus on how often the particles ``interact'' with each other by being in close proximity. Recently, Kiwi \emph{et al} \cite{Kiwi2022} applied such a multi-scale framework to hyperbolic graphs, where the location of vertices in hyperbolic space affects not only their relative positions to each other, but also how ``powerful'' the vertices are, i.e. how large their area of influence is.

Our focus is on \emph{geometric random graphs} embedded into euclidean space. Here, the vertices form a Poisson point process on $\R^d$ (or on the $d$-dimensional torus of volume $n$, denoted by $\mathbb{T}_n^d$) of intensity $\lambda>0$ and are each equipped with an independent, uniform on $(0,1)$ random mark. This mark can be thought of as an \emph{inverse weight}, with marks close to $0$ signifying powerful vertices and marks close to $1$ signifying weak vertices. Edges are then drawn between pairs of vertices with probability that depends on the pairs' relative locations and their respective  marks (cf. \Cref{ass:main}), with vertices close to each other and powerful vertices having a higher probability of forming an edge. Unlike in the hyperbolic case studied in  \cite{Kiwi2022}, we consider the mark of a vertex to be an innate property which does not evolve over time. In particular, although all vertices move around independently of their marks and each other, powerful vertices are relatively rare and it therefore takes much longer for areas that happen to be empty of such vertices to ``recover'' and return close to stationarity.

\paragraph{Soft Boolean model.}
Although we prove our results in more generality, we first state them for a motivating example, the \emph{soft Boolean model}. Let each vertex $x$ carry an independent, identically distributed random radius $R_x$, which we assume to be heavy-tailed, i.e. there exists $\gamma\in (0,1)$ such that 
\[
\P(R_x>r) \asymp r^{-d/\gamma}\ \text{as}\ r\to \infty.
\]
In the \emph{hard} version of the model two vertices are connected by an edge if the balls centered at the vertices locations with associated radii intersect. We consider a \emph{soft} version of this model, where an independent, identically distributed random variable $X(x,y)$ is assigned to each unordered pair of vertices $\{x,y\}$. Then, an edge is formed between two vertices $x$ and $y$ if and only if
\begin{equation}
\frac{\abs{x-y}}{R_x+R_y}\leq X(x,y).\label{eq:boolean}
\end{equation}
The choice $X(x,y)=1$, for all pairs $\{x,y\}$, corresponds to the hard version of the model. Writing $X$ for an independent identically distributed copy of $X(x,y)$, we assume $X$ to be heavy-tailed with decay
\[
\P(X>r) \asymp r^{-\delta d}\ \text{as}\ r\to \infty,
\]
for some $\delta>1$, leading to a relaxation of the condition to form an edge which can be interpreted in the following way. For any pair of vertices $x$ and $y$ we take a copy of their corresponding balls and expand those by multiplying both radii with $X(x,y)$. Then, we form an edge between the two vertices if and only if the expanded balls intersect.

\paragraph{Mobile soft Boolean model.}
We next introduce dynamics into the model. To that end, we let each vertex $x$ move over time independently as a Brownian motion. We do not modify the random radius $R_x$ over time, but we do resample the random variables $X(x,y)$ at unit interval times. One can think of this as the signal between two vertices getting ``boosted'' by a random amount, determining whether a connection between the pair can established or not. Importantly, whereas this boosting changes at discrete times, the rest of the model still evolves in continuous time, be it the movement of vertices or appearance/disappearance of edges according to \eqref{eq:boolean}.  See \Cref{fig:evolution} for an example of how edges can appear/disappear over time.

\begin{figure}[!ht]
	\begin{center}
	\begin{subfigure}[b]{.495\linewidth}
			\begin{tikzpicture}[scale=0.25, every node/.style={scale=0.7}]
							\node (A) at (-2,3)[circle, fill = black ,label={left:$x_1(s-)$}] {};
							\node (B) at (5.5,10)[circle, fill = black, label={left:$x_2(s-)$}] {};
							\draw[draw = lightgray, dashed] (A) circle (3.5cm);
							\draw[draw = lightgray, dashed] (B) circle (4cm);
							\draw[draw = red] (A) circle (4.2cm);
							\draw[draw = red] (B) circle (4.8cm);
							\draw[draw = red] (A) -- + (-60:4.2cm) node[right, text = red] {$X(\mathbf{x_1},\mathbf{x_2})R_{x_1}$};
							\draw[draw = red] (B) -- + (120:4.8cm) node[left, text = red] {$X(\mathbf{x_1},\mathbf{x_2})R_{x_2}$};
							%\draw[draw = red, thick] (A) to (B);
							%\node (Arne) at (15,-3)[label={\scriptsize Credit$\colon$ A. Grauer}] {};
			\end{tikzpicture}
			\subcaption{Situation just before time $s$.}\label{fig:evolution1}
	\end{subfigure}\hfill
	\begin{subfigure}[b]{.495\linewidth}
			\begin{tikzpicture}[scale=0.25, every node/.style={scale=0.7}]
							\node (A) at (-2,3)[circle, fill = black ,label={left:$x_1(s)$}] {};
							\node (B) at (5.5,10)[circle, fill = black, label={left:$x_2(s)$}] {};
							\draw[draw = lightgray, dashed] (A) circle (3.5cm);
							\draw[draw = lightgray, dashed] (B) circle (4cm);
							\draw[draw = red] (A) circle (5.25cm);
							\draw[draw = red] (B) circle (6cm);
							\draw[draw = red] (A) -- + (-60:5.25cm) node[right, text = red] {$X(\mathbf{x_1},\mathbf{x_2})R_{x_1}$};
							\draw[draw = red] (B) -- + (120:6cm) node[left, text = red] {$X(\mathbf{x_1},\mathbf{x_2})R_{x_2}$};
							\draw[draw = red, thick] (A) to (B);
							%\node (Arne) at (15,-3)[label={\scriptsize Credit$\colon$ A. Grauer}] {};
			\end{tikzpicture}
			\subcaption{Situation at time $s$.}
	\end{subfigure}
	
	\vspace{1cm}
	\begin{subfigure}[b]{.495\linewidth}
			\begin{tikzpicture}[scale=0.25, every node/.style={scale=0.7}]
							\node (A) at (-5,2)[circle, fill = black ,label={left:$x_1(t-)$}] {};
							\node (B) at (5.7,7)[circle, fill = black, label={left:$x_2(t-)$}] {};
							\node (C) at (-2,3)[circle, fill = lightgray ,label={right:$x_1(s)$}] {};
							\node (D) at (5.5,10)[circle, fill = lightgray, label={right:$x_2(s)$}] {};
							\pgfmathsetseed{3};
							\draw[decorate,decoration={random steps,segment length=1pt,amplitude=3pt}] (A) to (C);
							\draw[decorate,decoration={random steps,segment length=1pt,amplitude=3pt}] (B) to (D);
							\draw[draw = lightgray, dashed] (A) circle (3.5cm);
							\draw[draw = lightgray, dashed] (B) circle (4cm);
							\draw[draw = red] (A) circle (5.25cm);
							\draw[draw = red] (B) circle (6cm);
							\draw[draw = red] (A) -- + (-60:5.25cm) node[right, text = red] {$X(\mathbf{x_1},\mathbf{x_2})R_{x_1}$};
							\draw[draw = red] (B) -- + (120:6cm) node[left, text = red] {$X(\mathbf{x_1},\mathbf{x_2})R_{x_2}$};
							%\draw[draw = red, thick] (A) to (B);
							%\node (Arne) at (15,-3)[label={\scriptsize Credit$\colon$ A. Grauer}] {};
			\end{tikzpicture}
			\subcaption{Situation just before time $t:=s+1$.}
	\end{subfigure}\hfill
	\begin{subfigure}[b]{.495\linewidth}
			\begin{tikzpicture}[scale=0.25, every node/.style={scale=0.7}]
							\node (A) at (-5,2)[circle, fill = black ,label={left:$x_1(t)$}] {};
							\node (B) at (5.7,7)[circle, fill = black, label={left:$x_2(t)$}] {};
							\node (C) at (-2,3)[circle, fill = lightgray ,label={right:$x_1(s)$}] {};
							\node (D) at (5.5,10)[circle, fill = lightgray, label={right:$x_2(s)$}] {};
							\pgfmathsetseed{3};
							\draw[decorate,decoration={random steps,segment length=1pt,amplitude=3pt}] (A) to (C);
							\draw[decorate,decoration={random steps,segment length=1pt,amplitude=3pt}] (B) to (D);
							\draw[draw = lightgray, dashed] (A) circle (3.5cm);
							\draw[draw = lightgray, dashed] (B) circle (4cm);
							\draw[draw = red] (A) circle (4.55cm);
							\draw[draw = red] (B) circle (5.2cm);
							\draw[draw = red] (A) -- + (-60:4.55cm) node[right, text = red] {$X(\mathbf{x_1},\mathbf{x_2})R_{x_1}$};
							\draw[draw = red] (B) -- + (120:5.2cm) node[left, text = red] {$X(\mathbf{x_1},\mathbf{x_2})R_{x_2}$};
							%\draw[draw = red, thick] (A) to (B);
							%\node (Arne) at (15,-3)[label={\scriptsize Credit$\colon$ A. Grauer}] {};
			\end{tikzpicture}
			\subcaption{Situation at time $t:=s+1$.}
	\end{subfigure}
	
	\caption{An example of the evolution of $G_t$ illustrated with two vertices, for integer times $s$ and $t:=s+1$. The grey dashed circles have the random but fixed radii $R_{x_1}$ and $R_{x_2}$. 
	%In (c) and (d), the grey dots represent the locations of the two vertices at time $t_1$. 
	In (a) the random variable $X(\x_1,\x_2)$ has not been resampled yet and $\tfrac{\abs{x_1(s-)-x_2(s-)}}{R_{x_1}+R_{x_2}}\geq X(\x_1,\x_2)$, resulting in no edge between the vertices. In (b) the random variable $X(\x_1,\x_2)$ has been resampled, resulting in the edge between $\x_1$ and $\x_2$ being drawn. In (c), the vertices have moved far enough for the edge to no longer exist (note that $X(\x_1,\x_2)$ has not yet been resampled again). Finally, in (d), $X(\x_1,\x_2)$ has once more been resampled, with no edge being drawn as a consequence.}\label{fig:evolution}
	\end{center}
\end{figure}
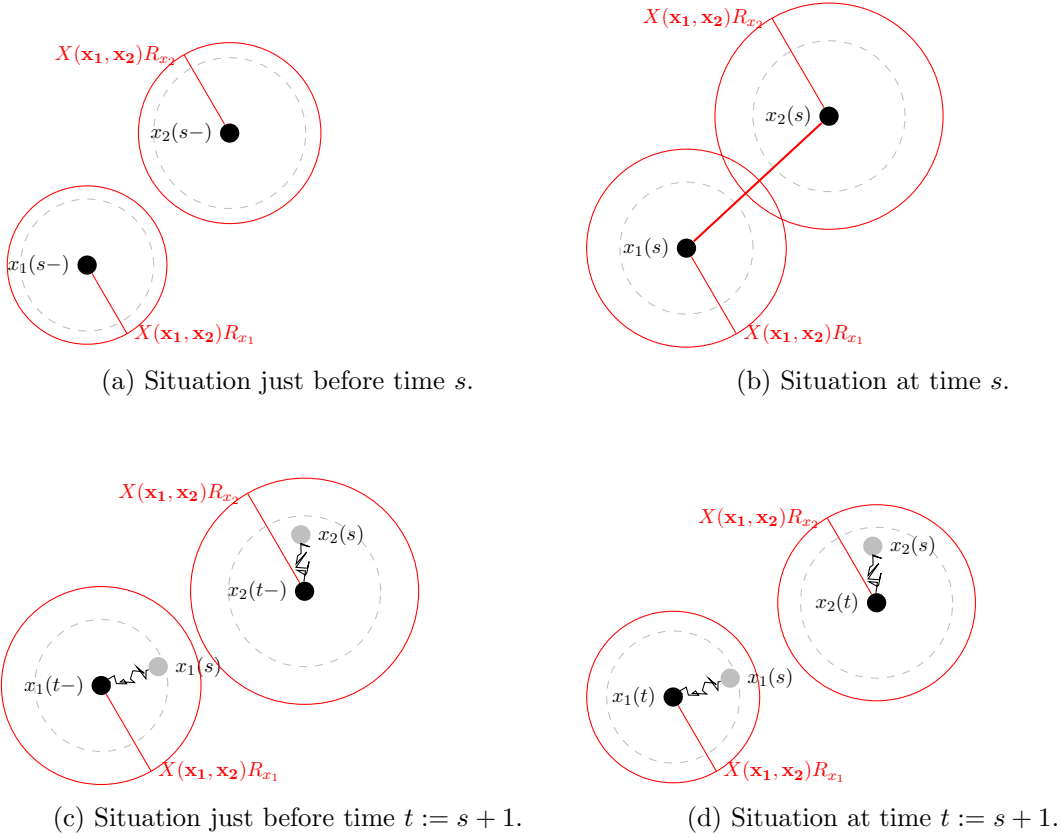

\paragraph{Broadcast time.}
In this setup we now study the spread of information. Consider first the mobile soft Boolean model on the $\mathbb{T}_n^d$, where at time $t=0$, an arbitrary vertex (without loss of generality located at the origin) starts broadcasting information. This information is immediately transmitted to all vertices in the same connected component (i.e. vertices for which a path to the origin vertex exists), upon which the vertices start broadcasting the information as well. We are interested in the time it takes for the information to reach every single vertex on the torus. 
%Recall that edges appear and disappear over time due to the motion of vertices, as well as the random noise arising from the resampling of the random variables $X(x,y)$. 
Denoting this  \emph{broadcast time} time by $T_{\text{bc}}$, we prove in \Cref{thm:flooding} that for $n$ large enough and if $\gamma>\frac{\delta}{\delta+1}$, information will reach every vertex on the torus $\mathbb{T}_n^d$ in poly-logarithmic time with high probability, for any positive particle intensity.

\paragraph{Percolation time.}
Next, consider the mobile soft Boolean model on $\R^d$. As above, let at time $t=0$ information start broadcasting from an arbitrary vertex, with the re-broadcasting rules the same as before. We are now interested in how long it takes for this information to reach the infinite connected component of the graph, which exists almost surely whenever $\gamma>\frac{\delta}{\delta+1}$, see \cite{Gracar2021}. We denote this by $T_{\text{perc}}$ and call it the \emph{percolation time}. We prove in \Cref{thm:percolation}, that with probability stretched exponentially close to $1$ in $t$, the percolation time is smaller than $t$.

While the above two results might not appear too surprising given what is known for the hard Boolean model with fixed radii covered by \cite{Peres2011}, note that the soft Boolean model introduces an additional source of randomness arising from the random variables $X(x,y)$ and more generally from the random occurrence of edges according to \Cref{ass:main} (cf. \Cref{sec:setup}). Furthermore, the fixed nature of the random radii over time introduces additional correlations into the model that could slow down the propagation of information, requiring significant work to overcome them.

\paragraph{Further examples of scale-free geometric random graphs}
In the \emph{age-dependent random connection model} each vertex carries a uniform on $(0,1)$ distributed birth time and two vertices $x$ and $y$ with birth times $t$ and $s$ are connected by an edge independently with probability
\begin{equation}
\varphi\big(\beta^{-1}(t\wedge s)^\gamma (t\vee s)^{1-\gamma}\abs{x-y}^d\big), \label{eq:conn_prob}
\end{equation}
where $\gamma\in (0,1)$, $\beta>0$ and $\varphi:(0,\infty)\to [0,1]$ is a non-decreasing function which we assume to satisfy $\varphi(r) \asymp r^{-\delta}$ as $r\to \infty$. This model emerges as a limit graph of a rescaled version of the \emph{age-based preferential attachment model} introduced in \cite{Gracar2019}. In this model vertices are added to the graph at rate of a Poisson process with unit intensity and placed on a torus of width one. A new vertex $x$ added at time $t$ forms an edge to each existing vertex $y$ with probability given by \eqref{eq:conn_prob}, where $s$ is the time vertex $y$ has been added to the graph. As $(t/s)^\gamma$ is the asymptotic order of the expected degree at time $t$ of a vertex with birth time $s$ the age-based preferential attachment model mimics the behaviour of spatial preferential attachment networks introduced in \cite{Jacob2015}.

This class of geometric random graphs has been studied recently and  exhibits a phase transition in the parameters $\gamma$ and $\delta$ such that these graphs are ultrasmall, i.e. two very distant vertices have a graph distance of doubly logarithmic order of their Euclidean distance, if and only if $\gamma>\frac{\delta}{\delta+1}$, as shown in \cite{Gracar2022a}. The same regime boundary depending on the parameters $\gamma$ and $\delta$ is shown to appear in the existence of a subcritical percolation phase by Gracar et al. \cite{Gracar2019}, giving a hint of a universal behaviour of these geometric random graphs that is remarkably different to the behaviour of spatial graph models investigated in \cite{Deijfen2013}. In this work we consider the class of geometric random graphs in their ultrasmall regime $\gamma>\frac{\delta}{\delta+1}$, evolving over time through vertex motion and edge updating.

\paragraph{Structure of the paper.}
The remainder of this paper is structured as follows. In \Cref{sec:setup} we formalise the setup and state \Cref{thm:flooding,thm:percolation}, the two main results of this paper. In \Cref{sec:technical}, we prove several technical results that will be required to prove the two theorems. More precisely, \Cref{subsec:vertexDensity} establishes the behaviour of the underlying vertices over time, while \Cref{subsec:evenlySpread} covers how under these conditions, well-behaved connected components can be constructed. \Cref{sec:proofs} then concludes by using these results to prove \Cref{thm:flooding,thm:percolation}.
\section{Setup and key results}\label{sec:setup}

\subsection{Geometric scale-free random graphs}

Let $\mathbb{T}_n^d$ the $d$-dimensional euclidean torus of volume $n$ and
let $\cG$ be a geometric random graph on a vertex set given by a Poisson point process $\mathcal{X}$ of intensity $\lambda$ on either $\R^d\times(0,1)$ or $\mathbb{T}_n^d\times(0,1)$. We write $\x=(x,u)$ for a vertex of the graph where we refer to $x$ as the location and $u$ as the mark of the vertex $\x$. For $\x_1,\dots,\x_n\in\R^d\times(0,1)$ (resp.\ $\mathbb{T}_n^d\times(0,1)$), denote by $\P_{\x_1,\dots,\x_n}$ the law of $\cG$ given the event that $\x_1,\dots,\x_n$ are points of the Poisson process. We consider geometric random graphs which satisfy the following assumption on the probability of the occurrence of edges in the graph. This assumption is given in terms of two parameters $\gamma\in (0,1)$ and $\delta>1$. For two vertices $\x,\y\in \mathcal{X}$ we write $\x\sim \y$ if there exists an edge between them.

\begin{assum}\label{ass:main}
	Given $\mathcal{X}$, edges are drawn independently of each other and there exist $\alpha,\kappa_1>0$ such that, for every pair of vertices $\x=(x,u), \y=(y,v)\in \mathcal{X}$, it holds
	\[
		\P_{\x,\y}(\x\sim \y)\geq \alpha\, \big(1\wedge \kappa_1 \, (u \wedge v)^{-\delta\gamma} \abs{x-y}^{-\delta d}\big)%\leq \kappa_2 \, (u\wedge v)^{-\delta\gamma} (u\vee v)^{\delta(\gamma-1)}\abs{x-y}^{-\delta d},
	\]
	where $|\cdot|$ is the euclidean distance on $\R^d$ (resp.\ $\T^d$). Furthermore, $\P_{\x,\y}(\x\sim \y)$ is such that it holds
	\[
		\E_{\x}|\{\y\in\X: \y\sim \x\}|<\infty.
	\]
\end{assum}
\begin{remark}
	The second condition of \Cref{ass:main} guarantees that we are working with finite degree graphs. An condition on $\P_{\x,\y}(\x\sim \y)$ that ensures this is, for example
	\[
		\P_{\x,\y}(\x\sim \y)\geq \kappa_2 \, (u\wedge v)^{-\delta\gamma} (u\vee v)^{\delta(\gamma-1)}\abs{x-y}^{-\delta d}
	\]
	for some constant $\kappa_2>0$. Note that both of the afore mentioned geometric random graphs, the soft Boolean model and the age-dependent random connection model satisfy this inequality, as well as the lower bound from \Cref{ass:main}.
\end{remark}

We will work with the Palm versions of $\X$ and $\mathcal{G}$. More precisely, we add to $\mathcal{X}$ a vertex $(0,U_0)$, where $U_0$ is an independent on the interval $(0,1)$ uniformly distributed random variable and denote by $\mathcal{G}_{\Palm}$ the resulting graph on $\mathcal{X}\cup \{(0,U_0)\}$ determined by the connection rules satisfying Assumption \ref{ass:main}. We denote the law of $\mathcal{G}_{\Palm}$ by $\P_{\Palm}$ and since $U_0$ is independent of the underlying Poisson point process $\mathcal{X}$, it holds $\P_{\Palm} = \P_{(0,u_0)}\mathd u_0$. We can think of $\P_{\Palm}$ of the law of $\mathcal{G}$ conditioned on the existence of a typical vertex, i.e. a vertex with typical mark, at the origin.

We will denote by $\mathcal{C}_\infty$ the infinite connected component of $\mathcal{G}$ should it exists. It is known (see \cite{Gracar2021}) that if $\gamma>\frac{\delta}{\delta+1}$, such a component exists almost surely and is unique for any particle intensity of $\lambda>0$. Furthermore, as shown in \cite{Jorritsma2023a}, all of the remaining connected components are of finite size. When restricted to a cube of volume $n$ on $\R^d$ or when working on the torus $\mathbb{T}_n^d$, \cite{Jorritsma2023a} also gives that with high probability and for sufficiently large $n$, there exists a unique \emph{giant component}, i.e. a connected component of linear size in the volume of the cube/torus, while all other connected components are of logarithmic size. In all of the above cases, we will denote by $\mathcal{C}(\x)$ the connected component containing the vertex $\x$ and by $|\mathcal{C}(\x)|$ its size, i.e. the number of vertices belonging to it. 

We will write $Q_K$ for a cube of volume $K$, centred around the origin. We might occasionally consider translates of $Q_K$, but since we will not be interested in them beyond the number that a larger cube can be tessellated into, we do not introduce any specific notation for them and the locations they are centred around. For $\x\in\R^d\times(0,1)$ (or equivalently $\mathbb{T}_n^d\times(0,1)$) will denote by $B(\x,r)$ the ball of radius $r\in\R_+$, centred at $x\in\R^d$. On the torus, $n$ will always be assumed to be significantly larger than $r$ so as to avoid any boundary effects.

\subsection{Mobile geometric scale-free random graphs}

We allow the vertices underpinning the graph to move over time. To that end, denote by $\X_0:=\X$ the intensity $\lambda$ Poisson point process of vertices with marks, at time $0$. Let each vertex of $\X_0$ move as an independent Brownian motion on $\R^d$; the mark of the vertex remains unchanged. Define $\X_t$ to be the point process obtained after the vertices of $\X_0$ have moved for time $t$. More precisely, for each $\x=(x,u)\in\X_0$, let $(\pi_{\x}(t))_{t\geq 0}$ be a standard Brownian motion on $\R^d$ and define $\X_t=\{(x+\pi_{\x}(t),u):\x\in \X_0\}$. We also introduce the shorthand notation $\x_t:=(x+\pi_{\x}(t),u)$ for $t\geq 0$. It is easy to check that $\X_t$ remains a Poisson point process of intensity $\lambda$ on $\R^d\times(0,1)$ (resp. $\mathbb{T}_n^d\times(0,1)$). That being said, $\X_t$ and $\X_0$ are not independent of each other and due to the fixed marks, even the classical mixing results do not apply directly. 

We will work for the remainder of this paper on the Palm version of the process, i.e. we assume that in $\X_0$ there exists a vertex with a mark uniformly distributed on $(0,1)$ at the origin. We will refer to this vertex as the \emph{origin vertex} and denote it by $\0$, even though almost surely, other than at time $0$, the vertex is not located at the origin of $\R^d$ (resp. $\mathbb{T}_n^d$).

In analogy to the static case, we define $\cG_t$ to be the general geometric random graph on the vertex set $\X_t$. Crucially, we still require the marginal distribution of $\cG_t$ to satisfy \Cref{ass:main} for every $t\geq 0$. We do not however assume in general that the events $\{\x_{t_1}\sim\y_{t_1}\}$ and $\{\x_{t_2}\sim\y_{t_2}\}$ are conditionally on $\x_{t_1}, \x_{t_2}, \y_{t_1}, \y_{t_2}$ independent of each other. Instead, we sample for every pair of vertices $\x,\y$ in $\X_0$ a sequence of uniform on $(0,1)$ random variables $(U_{\x,\y}^n)_{n\in \N_0}$ and form an edge between $\x$ and $\y$ at time $0$ if $U_{\x,\y}^0<\P_{\x,\y}(\x\sim \y)$. At time $t> 0$, we assume an edge exists between $\x_t$ and $\y_t$ if $U_{\x,\y}^n<\P_{\x_t,\y_t}(\x_t\sim \y_t)$ where $n$ is the unique integer satisfying $n\leq t<n+1$. By a standard coupling argument, this construction yields the desired marginal distributions of the graphs $\cG_t$. We do note that if $t_1,t_2\in[n,n+1)$ for some $n\in\N_0$, then even conditionally on $\x_{t_1}, \x_{t_2}, \y_{t_1}, \y_{t_2}$ the events $\{\x_{t_1}\sim\y_{t_1}\}$ and $\{\x_{t_2}\sim\y_{t_2}\}$ are not independent. If on the other hand such an $n\in\N_0$ does not exist (i.e., there exists some $n\in\N$ such that $t_1< n\leq t_2$ or $t_2< n\leq t_1$), then conditionally on $\x_{t_1}, \x_{t_2}, \y_{t_1}, \y_{t_2}$ the events $\{\x_{t_1}\sim\y_{t_1}\}$ and $\{\x_{t_2}\sim\y_{t_2}\}$ are in fact independent.

\begin{remark}
	The above defined edge updating mechanism is crucial to allow the dynamic graph to ``recover'' from unfavourable configurations of $\X$ and the edge marks at suitable time scales. One could (and we will in our future work) consider alternative edge mark updating mechanisms, such as updating individual edge marks after exponentially distributed random times, but this would (up to constants) not alter the evolution of the process. If the update times were to occur at longer time scales, we conjecture this would lead to a drastically different process evolution, as seen for example in \cite{Jacob2019,Jacob2022}.
\end{remark}
	We also extend the definitions of $\mathcal{C}_\infty$ and $\mathcal{C}(\x)$ to the mobile setup. To that end, we write $\mathcal{C}_\infty^t$ for the (unique) infinite connected component of $\mathcal{G}_t$ and $\mathcal{C}_t(\x)$ for the connected component containing the vertex $\x$ at time $t$ in $\mathcal{G}_t$. Note that the latter is again well defined both on $\R^d\times(0,1)$, as well as in the torus case $\mathbb{T}_n^d\times(0,1)$. As before, we denote by $|\mathcal{C}_t(\x)|$ its size.
	
\subsection{Main results}
Consider a mobile geometric scale-free random graph on the $d$-dimensional torus of volume $n$ and recall that we denote by $\lambda$ the intensity of the Poisson point process of vertices.
Suppose a message originates from the origin vertex $\0$, and every vertex that has already received the message broadcasts it continuously to all vertices in the same connected component. Denote by $T_{\text{bc}}$ the time until all vertices have received the message. 
Assume the graph satisfies \Cref{ass:main}. We then have the following result.

\begin{theorem}\label{thm:flooding}
	Let $\gamma>\frac{\delta}{\delta+1}$. On the $d$-dimensional torus of volume $n$ with $d\geq 1$, the broadcast time $T_{\text{bc}}$ is with high probability $O(\log n(\log\log n)^\epsilon)$ for any $\epsilon>0$ and any vertex intensity $\lambda>0$.
\end{theorem}

Consider now a mobile geometric scale-free random graph on $\R^d$ instead. As above, assume that a message originates at the origin vertex $\0$. We define the percolation time as
\[
	T_{\text{perc}}:=\inf\{t\geq 0:\exists \x\in\mathcal{C}_\infty^t \text{ s.t. }\x\sim\0\},
\]
that is the first time the origin vertex $\0$ belongs to the infinite connected component of $\mathcal{G}_t$. Suppose that the graph satisfies \Cref{ass:main}. We then have the following result.
\begin{theorem}\label{thm:percolation}
	Let $\gamma>\frac{\delta}{\delta+1}$. Then there exists a constant $c>0$ such that on the $\R^d$ with $d\geq 1$, the percolation time $T_{\text{perc}}$ satisfies
	\[
		\P(T_{\text{perc}}>t)\leq\exp\{-ct^{1/c}\},
	\]
	for any vertex intensity $\lambda>0$ and any $t>0$ sufficiently large.
\end{theorem}

\section{Technical results}\label{sec:technical}

We begin by proving several technical results that will give us control over how vertices of various power behave over time. This in turn will let us construct subgraphs with favourable geometric properties to ensure information can propagate further.
\subsection{Vertex density over time}\label{subsec:vertexDensity}
The proofs in this subsection are inspired by constructions introduced in \cite{Peres2011}, where the authors studied particle motion and how it affects the emergence of infinite components in the Poisson-Boolean graph model. We note that the introduction of (fixed) vertex marks into the problem presents the biggest difficulty in adapting the technique, in particular since the intensity of the most powerful vertices being considered at any given scale decreases exponentially fast with the size of the region in question. Despite this, we are still able to take advantage of the following mixing time/decoupling time result from \cite{Peres2011}.

\begin{prop}\label{prop:mixing}
	Fix $K>\ell>0$ and consider the cube $Q_{K^d}\subset \R^d$ tessellated into subcubes of side length $\ell$. Let $\Pi_0$ be an arbitrary point process at time $0$ that contains at least $\beta\ell^d$ vertices in each subcube of the tessellation for some $\beta>0$. Let $\Pi_{\Delta}$ be the point process obtained at time $\Delta$ from $\Pi_0$ after the vertices have moved according to a standard Brownian motion for time $\Delta$. Fix $\epsilon\in(0,1)$ and let $\Psi$ be an independent Poisson point process of intensity $(1-\epsilon)\beta$ on $\R^d$. Then there exists a coupling of $\Psi$ and $\Pi_{\Delta}$ and constants $c_1,c_2,c_3$ that depend on $d$ only, such that if $\Delta\geq \frac{c_1\ell^2}{\epsilon^2}$ and $K'\leq K-c_2\sqrt{\Delta\log\epsilon^{-1}}>0$, the vertices of $\Psi$ are a subset of the vertices $\Pi_{\Delta}$ inside the cube $Q_{K'}$ with probability at least
	\[
		1-\frac{K^d}{\ell^d}\exp\{-c_3\epsilon^2\beta\Delta^{d/2}\}.
	\]
\end{prop}

\begin{proof}
	The proof of this result can be found in \cite{Peres2011}. Variants of this proposition for weighted lattices and fractal lattices can be found in \cite{Gracar2019a} and \cite{Drewitz2023}.
\end{proof}

We fix now for the remainder of the paper two positive constants, $\epst\ll\frac{1}{\log 2}$ and $\theta\in(\frac{\log 2}{\gamma+\gamma/\delta},\log 2)$. Note that by definition $\theta\epst\ll 1$. We will also from here on out observe the process $\mathcal{G}_t$ in discrete \emph{time steps} $t=0,1,\dots$ and prove results about the vertex and graph configurations at these discrete observation times. Despite this, the underlying process remains continuous in time, including the appearance/disappearance of edges and motion of particles, both of which can occur between our observation times.

\begin{mydef}\label{def:t-alpha-dense}
	Let $t>0$ and $I_k:=(\frac{1}{2}e^{-(k+1)\theta d},\frac{1}{2}e^{-k\theta d})$, $k\in\{0,\dots,\lfloor (\epst\log t)/d\rfloor\}$ and $I_{-1}:=(\frac{1}{2},1)$.
	% OLD: We say a cube $Q\subset \R^d$ is \emph{$t$-$\alpha$-dense}, if for every $k\in\{-1,0,\dots,\lfloor (a\log t)/d\rfloor\}$, the vertices in $Q\times I_k$ contain as a subset an independent Poisson point process of intensity $1-\alpha$ on $\R^d\times I_k$.
	% NEW:
	We say a cube $Q\subset \R^d$ is \emph{$t$-$\alpha$-dense}, if for every $k\in\{-1,0,\dots,\lfloor (\epst\log t)/d\rfloor\}$, the locations of the vertices in $Q\times I_k$ contain as a subset an independent Poisson point process of intensity $(1-\alpha)\lambda|I_k|$ on $Q$, with marks restricted to $I_k$.
\end{mydef}
\begin{remark}
	Note in the above definition we do not require the marks to be uniformly distributed. In fact, it suffices for the marks to lie deterministically inside the corresponding intervals $I_k$. We will use this in the proof of the following proposition by disregarding the marks of the vertices in each individual layer and later simply treating them as if they are as small/large as their respective mark interval permits when applying various bounds that depend on the marks.
\end{remark}

\begin{prop}\label{prop:alpha_dense}
	Let $t>0$ be a sufficiently large integer and $\xi,\epsilon\in(0,1)$ two constants. Consider the cube $Q_{L^d}$, for $L=t$. Define for $i=0,\dots, t$ the events
	\[
		A_i=\{\text{at time }i\text{ the cube }Q_L\text{ is }t\text{-}\xi\text{-dense}\}.
	\]
	Then, for any vertex intensity $\lambda>0$, there exist two positive constants $c_1,c_2$ such that
	\[
		\P\left(\sum_{i=0}^{t-1}\1_{A_i}\geq(1-\epsilon)t\right)\geq 1-\exp\{-c_1t^{c_2}\}.
	\]
\end{prop}

\begin{proof}
	We will argue the proof across each of the $\lfloor (\epst \log t)/d\rfloor$ layers simultaneously, as we require the cube $Q_{L^d}$ to be $t$-$\xi$-dense for the same $(1-\epsilon)t$ time steps. We start with the ``thinnest'' layer $k^*:=\lfloor (\epst \log t)/d\rfloor$, i.e. the one in which the Poisson point process we are working with has the lowest intensity. The vertices with marks in this layer form a Poisson point process of intensity 
	\begin{align*}
		\lambda|I_{k^*}|&=\lambda\left(\frac{1}{2}e^{-(\epst \log t)\theta}-\frac{1}{2}e^{-((\epst \log t)/d+1)\theta d}\right)\\
		&>\frac{1}{4} \lambda t^{-\epst\theta}.
	\end{align*}
	We will from here on out disregard the actual marks of the vertices and simply use that they are all bounded to be in the interval $I_{k^*}$, which suffices to satisfy the condition of \Cref{def:t-alpha-dense}. 
	For the remaining layers $k\in\{0,\dots,k^*-1\}$ the intensity of the corresponding Poisson point processes with marks in the interval $I_k$ is strictly greater than that of layer $k^*$. Consequently, the bounds we will show for $k^*$ hold, \emph{ceteris paribus}, also for the other layers. 
	We will disregard the actual marks of these vertices as well, and only use that the marks belong to the appropriate mark interval $I_k$.

Let $\kappa=O(\log t)$ be the number of scales we will use in the multi-scale argument that follows. Next, set $L^2=L_1>L_2>\cdots>L_\kappa=L$. We will tessellate the cubes of side-length $L_i$ into subcubes of side length $\ell_i$, to which end we let $\ell_1>\ell_2>\cdots>\ell_\kappa$, with their precise values to be determined later (see \Cref{fig:spatialtessellation}). At this stage we only assume that $\lim_{t\to\infty}\ell_{i}=\infty$ for all $i\in\{1,\dots,\kappa\}$. 
	%$C t^{c_a}$ with $C$ some large constant that will be fixed momentarily and $c_a$ a small constant satisfying $a\theta<c_a d<1$.
	
	\begin{figure}[!h]
	\begin{center}
\begin{tikzpicture}
	\tikzstyle{every node}=[font=\small]
	\draw (0,0) rectangle (5,5);
	\draw[<->,dashed] (0,-1) -- (2.5,-1) node[above] {$L_{j-1}$} -- (5,-1);
	
	\draw (0.5,0.5) rectangle (4.5,4.5);
	\draw[<->,dashed] (0.5,-1.5) -- (2.5,-1.5) node[below] {$L_j$} -- (4.5,-1.5);
	
	\draw (6,0) rectangle (11,5);
	\draw[dashed] (6,1) -- (11,1);
	\draw[dashed] (6,2) -- (11,2);
	\draw[dashed] (6,3) -- (11,3);
	\draw[dashed] (6,4) -- (11,4);
	
	\draw[dashed] (7,0) -- (7,5);
	\draw[dashed] (8,0) -- (8,5);
	\draw[dashed] (9,0) -- (9,5);
	\draw[dashed] (10,0) -- (10,5);
	
	\draw[<->,dashed] (7,-1) -- (7.5,-1) node[above] {$\ell_{j-1}$} -- (8,-1);
	
	\draw[red,thick] (6.5,0.5) rectangle (10.5,4.5);
	\draw[dashed,red,thick] (6.5,0.5+4/5) -- (10.5,0.5+4/5);
	\draw[dashed,red,thick] (6.5,0.5+8/5) -- (10.5,0.5+8/5);
	\draw[dashed,red,thick] (6.5,0.5+12/5) -- (10.5,0.5+12/5);
	\draw[dashed,red,thick] (6.5,0.5+16/5) -- (10.5,0.5+16/5);
	
	\draw[dashed,red,thick] (6.5+4/5,0.5) -- (6.5+4/5,4.5);
	\draw[dashed,red,thick] (6.5+8/5,0.5) -- (6.5+8/5,4.5);
	\draw[dashed,red,thick] (6.5+12/5,0.5) -- (6.5+12/5,4.5);
	\draw[dashed,red,thick] (6.5+16/5,0.5) -- (6.5+16/5,4.5);
	
	\draw[<->,dashed,thick,red] (6.5+12/5,-1) -- (6.5+28/10,-1) node[above] {$\ell_{j}$} -- (6.5+16/5,-1);

\end{tikzpicture}
\caption{The spatial multi-scale recursion. Note that the difference $L_{j-1}-L_j$ is proportional to the value $\ell_{j-1}$ by \eqref{eq:elljmj} and \eqref{eq:Lj}, allowing us to apply \Cref{prop:mixing}.}\label{fig:spatialtessellation}
\end{center}
\end{figure}
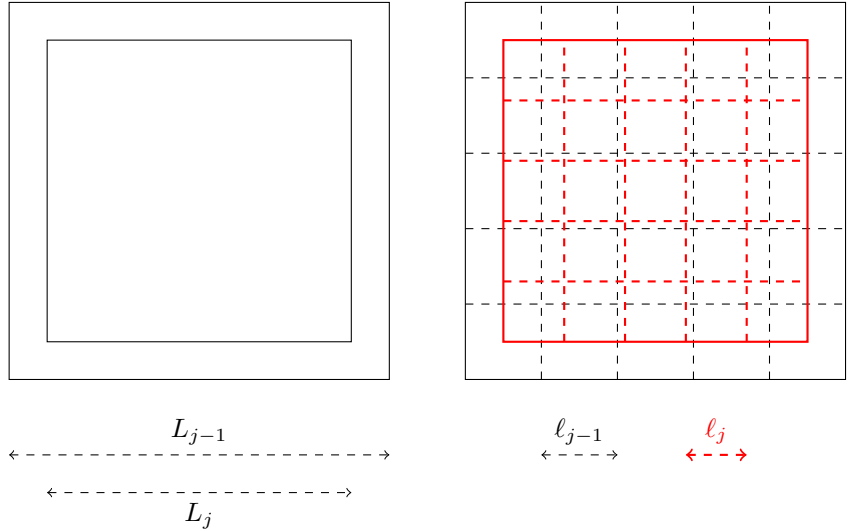
	
	We say that a subcube is \emph{good} at a given time step for scale $j$, if it contains at least 
	%$(1-\xi_j)\frac{1}{4}\lambda t^{-\epst\theta}\ell_j^d$ 
	\[
		(1-\xi_j)\lambda|I_k|
	\]
	vertices at the time in layer $k$, for $k\in\{0,\dots,k^*\}$, with $\xi_j$ chosen to satisfy
	\[
		\frac{\xi}{2}=\xi_1<\xi_2<\cdots<\xi_{\kappa}=\xi,
	\]
	and
	\[
		\xi_j-\xi_{j-1}=\frac{\xi}{2(\kappa-1)},\forall j.
	\]

	Let $D_1$ be the event that all subcubes of side-length $\ell_1$ inside the cube of side length $L_1$ are good for all time steps in $[0,t]$. 
	For a fixed subcube and a fixed time step, the number of vertices in a given layer $k$ is given by a Poisson random variable with mean 
	%$\frac{1}{4}\lambda t^{-\epst\theta}\ell_1^d$. 
	\(
		\lambda|I_k|\ell_1^d.
	\)
	Using \Cref{lemma:pchernoff}, we obtain that there are more than 
	%$(1-\xi_i)\frac{1}{4}\lambda t^{-\epst\theta}\ell_1^d$ 
	\(
		(1-\xi_i)\lambda|I_k|\ell_1^d
	\)
	vertices in this subcube in layer $k$ at the given time with probability larger than $1-\exp\{-\xi_1^2 \lambda |I_k|\ell_1^d/2\}$. 
	The number of subcubes of side-length $\ell_1$ inside $Q_{L_1^d}$ is $O(t^{2d})$ by the choice of $L_1$ and our assumption on $\ell_1$. Before taking a union bound across all of the layers, note that 
	\begin{equation}
		\lambda|I_k|\geq \lambda|I_{k^*}|>\tfrac{1}{4}\lambda t^{-\epst\theta},\quad\forall k\in\{0,1,\dots,k^*-1\}.\label{eq:kstar}
	\end{equation}
	Therefore taking the union bound across all the subcubes, all the time steps, and across all the $k^*$-many layers by using \eqref{eq:kstar}, we obtain that there exist two positive constants $c_1$ and $c_2$ such that
	\[
		\P(D_1)\geq 1-\lfloor (\epst \log t)/d\rfloor t^{2d+1 }\exp{\{-\xi_1^2 \lambda t^{-\epst\theta}\ell_1^d/8\}}\geq 1-\exp\{-c_1t^{c_2}\},
	\]
	where the last inequality holds for large enough $t$ if $\ell_1>C t^{\epsa}$. Here, $C$ is a large enough constant and $\epsa$ is a constant satisfying $\epst\theta<\epsa d<1$. Note that this assumption is in line with our previous assumption on the values $\ell_i$ increasing with $t$. We will later set $\ell_1=\Theta(\sqrt{t})$, which is in agreement with the above.
	
	We now proceed to smaller scales of the multi-scale argument. Let $s_j$ be the number of time steps considered for scale $j$. We start with $s_1=t$ so that at scale $1$, all time steps are considered. We will group the time steps into blocks of $m_j$ consecutive steps for each scale $j$, starting with $m_1=t$, i.e. we consider all the $t$ many steps at scale $1$ as a single block.
	
	At scale $j-1$, we will subdivide each interval $[b,b+m_{j-1})$ into four subintervals of length $m_j$ separated by $\Delta_{j-1}$ time steps (which we will determine below), of the form
	\[
		[b+k\Delta_{j-1}+(k-1)m_j,b+k\Delta_{j-1}+km_j),\,k\in\{1,2,3,4\},
	\]
	where $m_j=\frac{m_{j-1}-4\Delta_{j-1}}{4}$. Put into words, when going from scale $j-1$ to $j$, we split each existing interval into four subintervals of equal length, each subinterval preceded by a ``gap'' of $\Delta_{j-1}$ many steps (see \Cref{fig:timeintervals}). We will use these gaps to apply \Cref{prop:mixing}. Note that the above construction gives that $s_j=s_{j-1}(1-\frac{4\Delta_{j-1}}{m_{j-1}})$.
	Similarly to the spatial definition of good subcubes, we say that a time interval of scale $j$ is \emph{good}, if all the subcubes are good for scale $j$ during all of the time steps contained in the time interval.
	
	\begin{figure}[!h]
	\begin{center}
	\begin{tikzpicture}
		\tikzstyle{every node}=[font=\small]
		\draw  (0,0) rectangle  node {\small scale $j-1$} (10,0.5);
		\draw [<->,dashed] (0,0.75) -- (5,0.75) node[above] {$m_{j-1}$} -- (10,0.75);
		%\draw (5,1) node[above] {$m_{j-1}$};
		
		\draw [<->,dashed] (0,-1.95) -- (0.25,-1.95) node[below] {$\Delta_{j-1}$} -- (0.5,-1.95);
		\draw  (0.5,-2) rectangle  node {\small scale $j$} (2.5,-1.5);
		\draw [<->,dashed] (0.5,-1.25) -- (1.5,-1.25) node[above] {$m_j$} -- (2.5,-1.25);
		
		\draw [<->,dashed] (2.5,-1.95) -- (2.75,-1.95) node[below] {$\Delta_{j-1}$} -- (3,-1.95);
		\draw  (3,-2) rectangle  node {\small scale $j$} (5,-1.5);
		\draw [<->,dashed] (3,-1.25) -- (4,-1.25) node[above] {$m_j$} -- (5,-1.25);
		
		\draw [<->,dashed] (5,-1.95) -- (5.25,-1.95) node[below] {$\Delta_{j-1}$} -- (5.5,-1.95);
		\draw  (5.5,-2) rectangle  node {\small scale $j$} (7.5,-1.5);
		\draw [<->,dashed] (5.5,-1.25) -- (6.5,-1.25) node[above] {$m_j$} -- (7.5,-1.25);
		
		\draw [<->,dashed] (7.5,-1.95) -- (7.75,-1.95) node[below] {$\Delta_{j-1}$} -- (8,-1.95);
		\draw  (8,-2) rectangle  node {\small scale $j$} (10,-1.5);
		\draw [<->,dashed] (8,-1.25) -- (9,-1.25) node[above] {$m_j$} -- (10,-1.25);
	\end{tikzpicture}
	\caption{The temporal multi-scale recursion.}\label{fig:timeintervals}
	\end{center}
\end{figure}
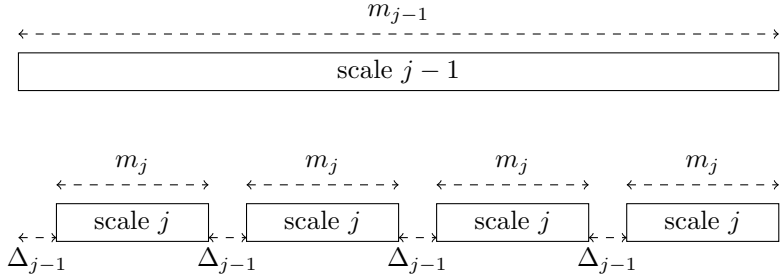
	
	Let now $0=\epsilon_1<\epsilon_2<\cdots<\epsilon_{\kappa}=\epsilon$ satisfy $\epsilon_j-\epsilon_{j-1}=\frac{\epsilon}{\kappa-1}$. We define for each scale $j\in \{2,\dots,\kappa-1\}$ the event $D_j$ to be the event that a fraction of at least $(1-\frac{\epsilon_j}{2})$ time intervals of scale $j$ are good. 
	We define $D_\kappa$ to be the event that for a fraction of at least $(1-\frac{\epsilon_\kappa}{2})$ time intervals of scale $\kappa$, the locations of the vertices in $Q_{L_\kappa^d}$ with marks in the layer $k$ contain as a subset an independent Poisson point process of intensity 
	%$(1-\xi_\kappa)\lambda t^{-\epst\theta}$,
	$(1-\xi_\kappa)\lambda|I_k|$,
	for all $k\in\{0,\dots,k^*\}$ simultaneously.
	
	In order to obtain the $(1-\epsilon)t$ many time steps from the claim of the proposition, we note that if $D_\kappa$ holds, for at least
	\begin{align*}
		(1-\frac{\epsilon_\kappa}{2})s_\kappa&=(1-\frac{\epsilon_\kappa}{2})s_1\prod_{j=1}^{\kappa-1}(1-\tfrac{4\Delta_j}{m_j})\\
		&\geq(1-\tfrac{\epsilon}{2})t\big(1-\sum_{j=1}^{\kappa-1}\tfrac{4\Delta_{j}}{m_j}\big)
	\end{align*}
	of the time steps the locations of the vertices in $Q_{L_\kappa^d}=Q_{L^d}$ with marks in the layer $k$ contain as a subset an independent Poisson point process of intensity $(1-\xi_\kappa)\lambda|I_k|$, for all $k\in\{0,\dots,k^*\}$ simultaneously. Setting $\frac{\Delta_j}{m_j}=\frac{\epsilon}{8\kappa}$ will therefore give the desired bound.
	
	We next proceed to set the value $\Delta_j$ to be sufficiently large to allow vertices to move over a distance $\ell_j$. To that end, we set $\ell_j$ via
	\begin{equation}\label{eq:deltadef}
		\Delta_j=C'\ell_j^2\kappa^2,
	\end{equation}
	with $C'$ some large enough constant. Combined with the above, this gives
	\begin{equation}\label{eq:elljmj}
		\frac{\ell_j^2}{m_j}=\frac{\epsilon}{8C'\kappa^3}
	\end{equation}
	and
	\[
		\ell_{j+1}=\ell_j\sqrt{\frac{1}{4}-\frac{\epsilon}{8\kappa}}.
	\]
	Since $m_1=t$, we get $\ell_1^2=\frac{\epsilon}{8C'\kappa^3}t$, %\leq\frac{\epsilon}{8C'}t$, 
	and since $\kappa=O(\log t)$, we also have $\ell_{\kappa}^d\geq C\log^3 t$ for any arbitrary constant $C$, if $t$ is large enough. Consequently, all of our preceding assumptions on the sequence $\ell_1,\dots,\ell_\kappa$ hold, in particular the assumption on $\ell_1$.
	
	We now argue that if at time $b':=b-\Delta_{j-1}$ all subcubes were good for scale $j-1$, and if $\ell_{\kappa}^d\geq C\log^3 t$ for $C$ large enough, then the interval $[b,b+m_j)$ of scale $j$ is good with probability larger than $1-\exp\{-c\lambda t^{-\epst\theta}\ell_j^d/\kappa^2\}$ for some positive constant $c$, uniformly across all events that are measurable with respect to the $\sigma$-algebra induced by particle behaviour up to and including time $b'$.
	
	To that end, let 
	\[
		E=\{\text{at time $b'$ all subcubes are good for the scale $j-1$}\}
	\]
	and let $F$ be any event measurable with respect to the above $\sigma$-algebra. If $E\cap F=\emptyset$, then 
	\[
		\P([b,b+m_j)\text{ is not good},E|F)=0,
	\]
	and the claimed bound holds. Therefore, let $E\cap F\neq\emptyset$ and let $\pi_{b'}^k$ be the point process of vertices in layer $k$, obtained at time $b'$ after conditioning on $E\cap F$. Fix now some time step $b^*\in[b,b+m_j)$. We can then bound
	\[
		\P(\text{at time $b^*$ not all subcubes are good for scale $j$}\,|\,E,F)
	\]
	as follows. We have by $E$ that all subcubes are good for the scale $j-1$ at time $b'$. We now set a constant $c_{\epsilon}$ such that $(1-c_{\epsilon})^2(1-\xi_{j-1})=1-\xi_j$, giving that $c_{\epsilon}=\theta(\xi_j-\xi_{j-1})$. We also choose a constant $c'$ and set $C'$ from \eqref{eq:deltadef} so that
	\begin{equation}\label{eq:Lj}
		L_j\leq L_{j-1}-c'\sqrt{\Delta_{j-1}\log\frac{1}{c_{\epsilon}}},
	\end{equation}
	which allows us to apply \Cref{prop:mixing} with $K=L_{j-1}$ and $K'=L_j$ across each of the $k^*+1$ many layers.
	By using \eqref{eq:kstar} a second time, we therefore obtain for all layers $k\in\{0,\dots,k^*\}$ a fresh Poisson point process $\Xi^k$ with intensity 
	%$(1-c_{\epsilon})(1-\xi_{j-1})\lambda t^{-\epst\theta}$
	$(1-c_{\epsilon})(1-\xi_{j-1})\lambda|I_k|$ 
	which can be coupled with $\pi^k_{b^*}$ in such a way that $\Xi^k$ is stochastically dominated by $\pi^k_{b'}$ for all $k\in\{0,\dots,k^*\}$ (i.e. the process after the vertices of $\pi^k_{b^*}$ have moved for time $b^*-b'$) inside $Q_{L_j}$ with probability at least
	\[
		1-\lfloor (\epst \log t)/d\rfloor\exp\big\{-c_1c_{\epsilon}^2\lambda(1-\xi_{j-1})t^{-\epst\theta}\ell_{j}^d\big\}
	\]%\todo{need more: adapt linear component proof to use logarithm correction.}
	for some constant $c_1>0$. Observe that the choice $L_1=t^2$ and $\kappa=O(\log t)$, together with \eqref{eq:deltadef} gives that it is always possible to choose $L_j$ satisfying \eqref{eq:Lj} and $L_\kappa=t$, provided $t$ is large enough. 
	
	Recall that $D_\kappa$ is defined differently than the remaining events $D_j$, $j<\kappa$ and we already have the desired bound for $D_\kappa$ from the above, so the following is only necessary for $j<\kappa$. A given subcube is good for scale $j$ at time $b^*$ if for all $k\in\{0,\dots,k^*\}$, $\Xi^k$ contains at least 
	%$(1-\xi_j)\frac{1}{4}\lambda t^{-\epst\theta}\ell_j^d$
	$(1-\xi_j)\lambda |I_k|\ell_j^d$  
	vertices in that subcube, which by our choice of $c_{\epsilon}$ and by applying \Cref{lemma:pchernoff}, occurs with probability at least 
	\[
		1-\lfloor (\epst \log t)/d\rfloor\exp\{-c_2c_{\epsilon}^2(1-c_{\epsilon})(1-\xi_{j-1}) \lambda t^{-\epst\theta}\ell_j^d\},
	\]
	where we have once more used \eqref{eq:kstar} and a union bound over the $k^*$-many layers.
	Using that $\ell_{\kappa}^d\geq C\log^3 t$ with $t$ large and taking a union bound across the time steps in $[b,b+m_j)$ and all subcubes of scale $j$, we obtain that there exists a positive constant $c$ such that
	\begin{equation}\label{eq:goodintervalbound}
		\P([b,b+m_j)\text{ is good, } E\,|\,F)\geq 1-\exp\{-c\lambda t^{-\epst\theta}\ell_j^d/\kappa^2\},
	\end{equation}
	with this bound trivially holding also for scale $\kappa$ by the previous comment about $D_\kappa$.
	
	Using \eqref{eq:goodintervalbound}, we now give an upper bound on the probability of the event $D_j^c\cap D_{j-1}$ for $j\geq 2$.
	If $D_{j-1}$ occurs, then by its definition there are at least $(1-\frac{\epsilon_{j-1}}{2})\frac{s_{j-1}}{m_{j-1}}$ intervals that are good for scale $j-1$. Going to scale $j$, these intervals subdivide into $4(1-\frac{\epsilon_{j-1}}{2})\frac{s_{j-1}}{m_{j-1}}$ subintervals of scale $j$ for us to consider. If the event $D_{j}^c$ is to hold, there should be fewer than $(1-\frac{\epsilon_{j}}{2})\frac{s_j}{m_j}$ intervals that are good for scale $j$. Let $w$ be the difference between these two values, i.e.
	\[
		w=\frac{s_j}{m_j}\left(\frac{\epsilon_j-\epsilon_{j-1}}{2}\right).
	\]
	Let $Z$ be the number of subintervals of $[b,b+m_j)$ of scale $j$ that are not good for scale $j$, but are such that at the time step $b-\Delta_{j-1}$ all subcubes were good for scale $j-1$. We obtain that if $D_{j-1}$ and $D_j^c$ hold simultaneously, $Z$ has to be at least $w$.
	
	$Z$ can be written as a sum of $s_j/m_j$ indicator random variables $\chi_r$, each one representing a time interval of scale $j$. Note that although the $\chi_r$ are not independent of each other, \eqref{eq:goodintervalbound} gives that the probability that $\chi_r=1$, given any realisation of the previous $r-1$ intervals, is smaller than $\rho_j:=\exp\{-c\xi^2t^{-\epst\theta}\lambda \ell_j^d/\kappa^2\}$ for some constant $c$. Consequently, $Z$ is stochastically dominated by a random variable $Z'$, distributed as a Binomial random variable with parameters $s_j/m_j$ and $\rho_j$. Using \Cref{lemma:bchernoff} leads to
	\begin{align*}
		\P(Z'\geq w)&=\P\big(Z'-\E[Z']\geq \tfrac{s_j}{m_j}(\tfrac{\epsilon_j-\epsilon_{j-1}}{2}-\rho_j)\big)\\
		&\leq \exp\big\{-\tfrac{s_j}{m_j}(\tfrac{\epsilon_j-\epsilon_{j-1}}{2})(\log(\tfrac{\epsilon_j-\epsilon_{j-1}}{2\rho_j})-1)\big\}.	
	\end{align*}
	Recall that $\epsilon_j-\epsilon_{j-1}=\frac{\epsilon}{\kappa-1}$ and $-\log(\rho_j)=\theta(\xi^2\lambda t^{-\epst\theta}\ell_j^d/\kappa^2)$. Furthermore recall that $\ell_j^d\geq \ell_\kappa^d\geq C\log^3 t$ and $\kappa=O(\log t)$, which together gives that for $t$ large enough, we can find a positive constant $c$ for which
	\[
		\P(Z'\geq w)\leq\exp\big\{-c\lambda\xi^2 s_j\tfrac{t^{\epst\theta}\ell_j^d}{m_j}\tfrac{\epsilon}{\kappa^3}\big\}.
	\]
	Recall next that by \eqref{eq:elljmj} we have $\frac{\ell_j^2}{m_j}=\frac{\epsilon}{8C'\kappa^3}$ and note that by its definition $s_{j-1}=\theta(t)$ for all $j$. We therefore obtain
	\[
		\P(Z'\geq w)\leq\exp\big\{-c\lambda\tfrac{\epsilon^2\xi^2}{\kappa^6}t^{-\epst\theta}\ell_j^{d-2}t\big\}
	\]
	for some (new) constant $c$.
	Together, this yields that 
	\[
		\P(D_j^c\cap D_{j-1})\leq \exp\big\{-\tfrac{ct^{c_3}}{\kappa^6}\big\},
	\]
	where we used that $\ell_j\geq \ell_\kappa$ and $\ell_\kappa^d\geq C\log^3 t$, and $c,c_3>0$ are some positive constants.
	
	We are now ready to conclude the proof. To that end, we first want to bound the probability $\P(D_{\kappa}^c)$. We have that 
	\[
		\P(D_\kappa^c)\leq \P(D_\kappa^c\cap D_{\kappa-1})+\P(D_{\kappa-1}^c)
	\]
	and applying this recursively, it follows that
	\[
		\P(D_\kappa^c)\leq\sum_{j=2}^\kappa \P(D_j^c\cap D_{j-1})+\P(D_1^c).
	\]
	Each term in this sum can be bounded from above by $\exp\{-\frac{c}{\kappa^6}t^{c_3}\}$ and the last term by $\exp\{-c_1 t^{c_2}\}$, which gives the stated claim, using that $\kappa=O(\log t)$. 
\end{proof}

\subsection{Evenly spread subgraphs}\label{subsec:evenlySpread}

\begin{mydef}\label{def:large_component}
	We call a finite connected subgraph of $\cG$ contained inside $Q_K$ an \emph{evenly spread subgraph of $\cG$ inside $Q_K$}, if it contains at least $b\cdot K$ vertices for some constant $b>0$ and if every subcube of $Q_K$ of the form \(\times_{i=1}^d(2^{k_p} v_i,(2^{k_p}+1)v_i)\), \(v_i\in\Z\), $k_p=\lfloor \epst\log K/d\rfloor$ contains a vertex with mark smaller than \(\frac{1}{2}e^{-k_p\theta d}\) belonging to the evenly spread component. We call these vertices the \emph{bottom vertices} of the evenly spread component.
\end{mydef}

\begin{prop}\label{prop:large_component}
	Fix the vertex intensity $\lambda>0$, fix $K$ large enough and consider the cube $Q_K\subset \R^d$. Assume that $Q_K$ is $K$-$\alpha$-dense for some $\alpha>0$ and let $\gamma>\frac{\delta}{\delta+1}$. Then, there exists $\epsilon>0$ such that for $K$ large enough, there exists an evenly spread subgraph of $\cG$ inside $Q_K$ with probability at least
	\[
		1-\exp\{-K^\epsilon\}.
	\]
	%Furthermore, at least $b\cdot K^d$ of the vertices in this component have marks larger than $K^{-d}$.
\end{prop}

\begin{proof}
	Recall the values $\epst\in(0,\frac{1}{\log 2})$ and $\theta\in(\frac{\log 2}{\gamma+\gamma/\delta},\log 2)$. Set $\epsilon_1>0$ small enough that $\theta>\frac{\epsilon_1+\log 2}{\gamma+\gamma/\delta}$.
		Let $n_p=\lfloor K^{(1-\epst \log 2)/d}\rfloor$ and $k_p=\lfloor \epst\log K/d\rfloor$. For $k=0,\dots,k_p$ define
	\[
		V_k:=\{0,\dots,n_p 2^{k_p-k}-1\}^d
	\]
	and
	\[
		A_{k,\mathbf{v}}:=\bigtimes_{i=1}^d\big(2^k v_i,2^k(v_i+1)\big)\;\text{for }k=0,\dots k_p\text{ and }\mathbf{v}=(v_1,\dots,v_d)\in V_k.
	\]
	
	For each $k=0,\dots,k_p$, the cubes $\{A_{k,\mathbf{v}}:\mathbf{v}\in V_k\}$ give a tessellation of $Q_K$ into $(n_p 2^{k_p-k})^d$ cubes of volume $2^{kd}$, with the finest tessellation at $k=0$ and the coarsest at $k=k_p$. Note that the cubes across layers are nested inside each other in a tree-like structure with every $A_{k+1,\v}$ tessellating into the 4 cubes $A_{k,2\v+\mathbf{e}}:\mathbf{e}\in\{0,1\}^d$.
	
	Define
	\[
		B_{k,\v}:=A_{k,\v}\times\big(\tfrac{1}{2}e^{-(k+1)\theta d},\tfrac{1}{2}e^{-k\theta d}\big)\;\text{for }k=0,\dots k_p\text{ and }\v=(v_1,\dots,v_d)\in V_k.
	\]
	We use the parameter $k=0,\dots,k_p$ to signify the layer of the boxes $\{B_{k,\v}:\v\in V_k\}$, which defines the range of the marks of points of $\X$ inside the boxes and the level of the coarseness of the tessellation of space. Note that the ranges of the marks correspond with the ones given by \Cref{def:t-alpha-dense}. Note further that at the ``lowest layer'', i.e. for $k=k_p$, the boxes have width of order $K^{\epst\log 2}$ and the corresponding mark range is of order $K^{-\epst\theta \log 2}$. Due to the nested structure of the boxes $A_{k,\v}$ and disjoint nature of the intervals $\big(\tfrac{1}{2}e^{-(k+1)\theta d},\tfrac{1}{2}e^{-k\theta d}\big)$, the system of boxes $\{B_{k,\v}:k=0,\dots,k_p,\v\in V_k\}$ has the structure of a $2^d$-regular tree when treating $B_{k+1,\v}$ as the common parent of the $2^d$ boxes $B_{k,2\v+\mathbf{e}}$ with $\mathbf{e}\in\{0,1\}^d$; see \Cref{fig:box_nested} for an example.
	
	\begin{figure}[h]
\begin{center}
\begin{tikzpicture}[scale=0.4, every node/.style={scale=0.6}]
\draw (0,4.5)  -- (0,4.5) node[left] {$\frac{1}{2}$} --  (0, -3) node[below left] {$0$};
\draw (0,-3)  -- (24.5,-3);
%\draw[dashed] (0,-1.25) rectangle ++(16,0.75);
%\draw[dashed] (16,-1.25) -- (24,-1.25);
\node at (4,-1) {$\vdots$}; 
\node at (12,-1) {$\vdots$}; 
\node at (20,-1) {$\vdots$}; 
\draw (0,-0.5) rectangle ++(8,1); \node at (4,-0.1) {$B_{3,1}$};
\draw (8,-0.5) rectangle ++(8,1); \node at (12,-0.1) {$B_{3,2}$};
\draw (16,-0.5) rectangle ++(8,1); \node at (20,-0.1) {$B_{3,3}$};
\draw (0,0.5) rectangle ++(4,1.5); \node at (2,1) {$B_{2,1}$};
\draw (4,0.5) rectangle ++(4,1.5); \node at (6,1) {$B_{2,2}$};
\draw (8,0.5) rectangle ++(4,1.5); \node at (10,1) {$B_{2,3}$};
\draw (12,0.5) rectangle ++(4,1.5); \node at (14,1) {$B_{2,4}$};
\draw (16,0.5) rectangle ++(4,1.5); \node at (18,1) {$B_{2,5}$};
\draw (20,0.5) rectangle ++(4,1.5); \node at (22,1) {$B_{2,6}$};
\draw (0,2) rectangle ++(2,2.5); \node at (1,3) {$B_{1,1}$};
\draw (2,2) rectangle ++(2,2.5); \node at (3,3) {$B_{1,2}$};
\draw (4,2) rectangle ++(2,2.5); \node at (5,3) {$B_{1,3}$};
\draw (6,2) rectangle ++(2,2.5); \node at (7,3) {$B_{1,4}$};
\draw (8,2) rectangle ++(2,2.5); \node at (9,3) {$B_{1,5}$};
\draw (10,2) rectangle ++(2,2.5); \node at (11,3) {$B_{1,6}$};
\draw (12,2) rectangle ++(2,2.5); \node at (13,3) {$B_{1,7}$};
\draw (14,2) rectangle ++(2,2.5); \node at (15,3) {$B_{1,8}$};
\draw (16,2) rectangle ++(2,2.5); \node at (17,3) {$B_{1,9}$};
\draw (18,2) rectangle ++(2,2.5); \node at (19,3) {$B_{1,10}$};
\draw (20,2) rectangle ++(2,2.5); \node at (21,3) {$B_{1,11}$};
\draw (22,2) rectangle ++(2,2.5); \node at (23,3) {$B_{1,12}$};
\end{tikzpicture}
\caption{Sketch of the structure of the boxes $B_{k,\v}$ in dimension one. The y-axis represents the mark of the vertices and the x-axis the location.}\label{fig:box_nested}
\end{center}
\end{figure}
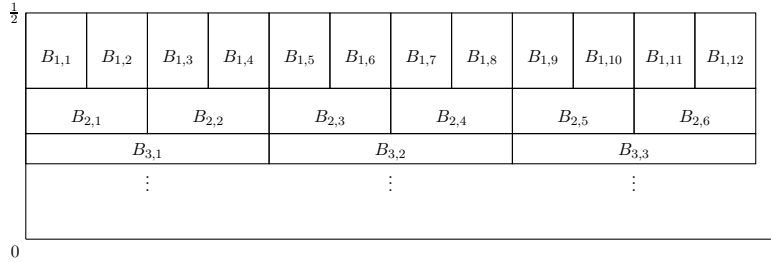
	
	Due to the assumption of $K$-$\alpha$-density, the number of vertices in any box $B_{k,\v}$ stochastically dominates a Poisson random variable with parameter
	\[
		(1-\alpha)2^{kd}\lambda \big(\tfrac{1}{2}e^{-k\theta d}-\tfrac{1}{2}e^{-(k+1)\theta d}\big)>(1-\alpha)\lambda ce^{kd(\log 2-\theta)}
	\]
	for some constant $c>0$ not depending on $k$ and $\v$. If we therefore set $\epsilon_2=\log 2-\theta>0$, it follows that
	\begin{equation}\label{eq:nonempty}
		\P(B_{k,\v}\text{ is not empty})\geq 1-\exp\big\{(1-\alpha)\lambda ce^{kd\epsilon_2}\big\}.
	\end{equation}
	On the event that $B_{k,\v}$ is non-empty, we denote by $\x_{k,\v}$ an arbitrarily chosen vertex (say, with the smallest mark) in the box. 
	
	We will use the boxes $B_{k,\v}$ to build the linear sized component. We will then use the even spatial distribution of these vertices to ensure that the vertices can form connections with each other through connectors. To this end, we colour all of the vertices with marks greater than $1/2$. Choose $\epsilon_3\in(0,\theta\gamma\wedge \delta\epsilon_1)$ such that $\sum_{k=0}^\infty e^{-kd(\theta\gamma\wedge\delta\epsilon_1-\epsilon_3)}$ converges and colour the points of $\X$ on $\R^d\times[\frac{1}{2},1)$ by the colour set $\N$ independently and such that the points with colour $k\in\N$ form a Poisson point process $\X^k$ on $\R^d\times[\frac{1}{2},1)$ with intensity proportional to $e^{-kd(\theta\gamma\wedge\delta\epsilon_1-\epsilon_3)}$. We now use this colouring to assign the coloured vertices different roles. For $k=0,\dots,k_p$ and $\v\in V_k$, we denote as
	\begin{itemize}
		\item $C_{k,\v}=\X^k\cap A_{k,\v}\times[\frac{1}{2},1)$ the \emph{potential connectors} (if $B_{k,\v}$ is non-empty).
	\end{itemize}
	We will use $C_{k,\v}$ to ensure with high probability that the vertices we consider belong to the same connected component.  Note that the intensity of of $C_{k,\v}$ is decreasing in $k$, since it is much easier for vertices with small marks (i.e. large $k$) to successfully form connections as opposed to vertices with large marks (i.e. small $k$), which require a larger number of candidates to compensate for the lower probability of forming edges.
	
	Using that $\theta\gamma<\log 2$, there exists $c>0$ such that, for all $k=0,\dots,k_p$, on the event that $B_{k,\v}$ is not empty, $ct_{k,\v}^{-\gamma/d}<2^k$, where $t_{k,\v}$ is the mark of $\x_{k,\v}$. Consequently, for each $k=0,\dots,k_p$, the volume of $B(\x_{k,\v},ct_{k,\v}^{-\gamma/d})\cap A_{k,\v}$ represents a positive proportion $\rho>\frac{1}{2^d}$ of the volume of the ball $B(\x_{k,\v},ct_{k,\v}^{-\gamma/d})$. 
	%This has two consequences. First, by the same argument as in the proof of \Cref{lemma:twoconn}, given $B_{k,\v}$ is not empty, the number of neighbours of $\x_{k,\v}$ in $N_{k,\v}$ is Poisson-distributed with parameter larger than $ce^{-kd(\theta\gamma\wedge\delta\epsilon_1-\epsilon_3)}e^{kd\theta\gamma}>ce^{kd\epsilon_3}$ for some constant $c>0$ that does not depend on $k$.
	%We denote by $\operatorname{Linear}(k,\v)$ the event that $B_{k,\v}$ is not empty and that $\x_{k,\v}$ has at least $S$ neighbours in $N_{k,\v}$. Then, using a simple Chernoff bound for Poisson random variables, there exists a constant $c>0$ and $k_0$ sufficiently large and not depending on $K$ such that for all $k\geq k_0$, $\v\in V_k$ it holds
	%\begin{equation}\label{eq:linear}
	%	\P(\operatorname{Linear}(k,\v)|B_{k,\v}\text{ is not empty})\geq 1-\exp\{-ce^{kd\epsilon_3}\}
	%\end{equation}
	This has the following consequence on how two vertices can connect. Given that the box $B_{k+1,\v}$ and one of its ``children'' $B_{k,2\v+\mathbf{e}}$ (with $\mathbf{e}\in\{0,1\}^d$) are not empty, the corresponding vertices $\x_{k+1,\v}$ and $\x_{k,2\v+\mathbf{e}}$ lie at distance at most $\sqrt{d}2^{k+1}$ and both have marks smaller than $\frac{1}{2}e^{-k\theta d}$. Using \Cref{lemma:twoconn}, there exists a constant $c>0$ such that the number of vertices in $C_{k,\v}$ which form an edge to both $\x_{k+1,\v}$ and $\x_{k,2\v+\mathbf{e}}$ is Poisson-distributed with parameter greater than
	
	\begin{equation*}
		c\lambda e^{kd(\theta\gamma\wedge\delta\epsilon_1-\epsilon_3)}t_{k,2\v+\mathbf{e}}^{-\gamma}(1\wedge t_{k+1,\v}^{-\gamma\delta}(|x_{k,2\v+\mathbf{e}}-x_{k+1,\v}|+t_{k,2\v+\mathbf{e}}^{\gamma/d})^{-d\delta})>c\lambda e^{kd\epsilon_3},
	\end{equation*}
	where $c$ changes throughout the calculation, but does no depend on $k$ and $\v$. We have used above that $\theta>\frac{\epsilon_1+\log 2}{\gamma+\gamma/\delta}$. Write $\{\x_{k+1,\v}\leftrightarrow\x_{k,2\v+\mathbf{e}}\}$ for the event that the boxes $B_{k+1,\v}$ and $B_{k,2\v+\mathbf{e}}$ are not empty and the vertices $\x_{k+1,\v}$ and $\x_{k,2\v+\mathbf{e}}$ are connected via a vertex in $C_{k,2\v+\mathbf{e}}$. Then, we have due to the above, for all $k\in \N$ and $\v\in V_k$ that
	\begin{equation}\label{eq:twosteps}
		\P(\x_{k+1,\v}\leftrightarrow\x_{k,2\v+\mathbf{e}}|B_{k+1,\v}\text{ and }B_{k,2\v+\mathbf{e}}\text{ are not empty})\geq 1-\exp\{-c\lambda e^{kd\epsilon_3}\}.
	\end{equation}
	
	We now proceed to use \Cref{eq:nonempty,eq:twosteps} together with the tree-like structure of the boxes to prove that with high probability, there exist of order $K$ distinct vertices belonging to the same connected component. We will do this in two steps. First, we will show that the vertices in the most powerful layer $k_p$ form a connected subgraph containing $n_p^d$ distinct vertices, where each box $B_{k_p,\v}$ contains one of these vertices. Second, we will use that that each box $B_{k_p,\v}$ for $\v\in V_{k_p}$ represents the root of a $2^d$-regular tree. We will show that the trees resulting only from boxes that are not empty and connected to their ``parent'' box through a connector are percolated $2^d$-regular trees with a depth of order $k_p$ and containing $2^{k_p d}$ vertices. Since there are $n_p^d$ many trees like this, the resulting connected subgraph of $\cG$ contains at least of order $K^d$ distinct vertices.
	
	To keep notation concise, we redefine the labelling of the boxes of layer $k_p$. Let 
	\[
		\sigma:\{0,\dots,n_p^d-1\}\rightarrow V_k
	\]
	be any bijection such that $B_{k_p,\sigma(0)}=B_{k_p,\0}$ and the boxes $B_{k_p,\sigma(i)}$ and $B_{k_p,\sigma(i+1)}$ are adjacent to each other for $i=0,\dots,n_p^d-2$. We also relabel the corresponding vertices with the smallest mark inside each box, so that on the event that $B_{k_p,\sigma(i)}$ is not empty, $\x_{k_p,\sigma(i)}$ is the vertex with the smallest mark inside it. We say $B_{k_p,\sigma(0)}$ is good if and only if the box is not empty. By \eqref{eq:nonempty}, there exists $c>0$ such that
	\[
		\P(B_{k_p,\sigma(0)}\text{ is good})\geq 1-\exp\{-ce^{k_p d\epsilon_2}\}.
	\]
	For $i=0,\dots,n_p^d-2$, we say $B_{k_p,\sigma(i+1)}$ is good if
	\begin{enumerate}[(i)]
		\item $B_{k_p,\sigma(i)}$ is good,
		\item $B_{k_p,\sigma(i+1)}$ is not empty, and
		\item $\x_{k_p,\sigma(i+1)}$ and $\x_{k_p,\sigma(i)}$ are connected via a connector in $C_{k_p,\sigma(i+1)}$.
	\end{enumerate}
	Otherwise $B_{k_p,\sigma(i+1)}$ is considered bad. Let $\epsilon_4<\epsilon_2\wedge\epsilon_3$. By \Cref{eq:nonempty,eq:twosteps}, there exists $c>0$ such that for $i=0,\dots,n_p^d-2$ we have
	\[
		\P(B_{k_p,\sigma(i+1)}\text{ is good}|B_{k_p,\sigma(i)}\text{ is good})\geq 1-\exp\{-c\lambda e^{k_p d\epsilon_4}\}.
	\]
	This yields that
	\begin{equation}\label{eq:baseGood}
		\P(B_{k_p,\v}\text{ is good for all }\v\in V_{k_p})\geq 1-n_{k_p}^d\exp\{-c\lambda e^{k_p d\epsilon_4}\}\geq 1-K^{1-\epst\log 2}\exp\{-c\lambda K^{\epst\epsilon_4}\}.
	\end{equation}
	
	We now proceed to the remaining layers. For $\v\in\N^d$, denote by $\lfloor\frac{\v}{2}\rfloor$ the vector $(\lfloor\frac{v_1}{2}\rfloor,\dots,\lfloor\frac{v_d}{2}\rfloor)$. With this notation, for each $k=0,\dots,k_p-1$ and $\v\in V_k$, the parent box of $B_{k,\v}$ is given by $B_{k+1,\lfloor\frac{\v}{2}\rfloor}$. We now say, for each $k=0,\dots,k_p-1$ and $\v\in V_k$, that the box $B_{k,\v}$ is good if
	\begin{enumerate}[(i)]
		\item $B_{k+1,\lfloor\frac{\v}{2}\rfloor}$ is good,
		\item $B_{k,\v}$ is not empty, and
		\item $\x_{k,\v}$ and $\x_{k+1,\lfloor\frac{\v}{2}\rfloor}$ are connected via a connector in $C_{k,\v}$.
	\end{enumerate}
	As before, we otherwise say $B_{k,\v}$ is bad. Just like in the case of $k_p$, equations \eqref{eq:nonempty} and \eqref{eq:twosteps} imply the existence of a constant $c>0$ such that for $k=0,\dots,k_p-1$, $\v\in V_k$ and $\mathbf{e}\in\{0,1\}^d$, such that
	\begin{equation}
		\P(B_{k,2\v+\mathbf{e}}\text{ is good}|B_{k+1,\v}\text{ is good})\geq 1-\exp\{-c\lambda e^{k d\epsilon_4}\}.
	\end{equation}
	Importantly, given that $B_{k+1,\v}$ is good, the events that $B_{k,2\v+\mathbf{e}}$ is good are all independent of each other and of any other box on this layer, due to only depending on disjoint subsets of $\X$. Consequently, the number of good children of $B_{k+1,\v}$, given that this box is good, is a binomially distributed random variable with parameters $2^d$ and $p_k\geq 1-\exp\{-c\lambda e^{kd\epsilon_4}\}$. Furthermore, given the good boxes in the layer $k+1$, the number of good children of each of those good boxes are independent of each other by the above observation. As such, if we denote by $|B_k|$ the number of good boxes in layer $k$, given $|B_{k+1}|$, $|B_k|$ is binomially distributed with parameters $2^d|B_{k+1}|$ and $p_k$.
	
	We can now estimate the number of good boxes in each layer and show that with high probability, sufficiently many are good. For $k=0,\dots,k_p-1$, denote by $E_k:=\{|B_k|>2^d(1-k^{-2})|B_{k+1}|\}$ the event that layer $k$ has sufficiently many good boxes when compared to layer $k+1$. Furthermore, denote by $E_{k_p}$ the event that all boxes in layer $k_p$ are good. Then, given the event $E_k\cap E_{k+1}\cap\dots\cap E_{k_p}$ we have that
	\[
		|B_k|>|B_{k_p}|\prod_{i=k}^{k_p-1}2^d(1-i^{-2})>cn_{p}^d2^{d(k_p-k-1)}>c2^{-kd}K,
	\]
	where $c=\prod_{i=1}^\infty(1-i^{-2})>0$. It is therefore sufficient to show that there exists a $k_0$ such that $E_{k_0}\cap\dots E_{k_p}$ holds with high probability. We choose $k_0$ large enough that $1-\exp\{ce^{k_0 d\epsilon_4}\}>1-k_0^{-2}$. Then, using a simple Chernoff bound for binomially distributed random variables we get
	\[
		\P\big(E_k^c\,\big|\,|B_{k+1}|\big)\leq \exp\left\{-\frac{2^{d-1}|B_{k+1}|c\lambda e^{kd\epsilon_4}}{k^2}\right\}
	\]
	for all $k\geq k_0$. Consequently, there exists $c>0$ such that
	\[
		\P(E_k|E_{k+1}\cap\dots\cap E_{k_p})\geq 1-\exp\left\{-c\lambda \frac{e^{kd\epsilon_4}2^{-kd}K^d}{k^2}\right\}\geq1-\exp\left\{-c\lambda \frac{K^{\epst d\epsilon_4}}{(\log K)^2}\right\}, 
	\]
	provided $k_p$ is sufficiently large.
	Together with \eqref{eq:baseGood}, it follows that
	\[
		\P(E_{k_0}\cap\dots\cap E_{k_p})\geq\P(E_{k_p})\left(1-\lfloor \epst \log K\rfloor\exp\left\{-c\lambda\frac{K^{\epst d\epsilon_4}}{(\log K)^2}\right\}\right)\geq 1-\exp\{-K^\epsilon\}
	\]
	for some $\epsilon>0$ not depending on $K$, as long as $K$ and consequently $k_p$ are large enough. As $E_{k_0}\cap\dots\cap E_{k_p}$ implies the existence of, up to a constant, at least $2^{-k_0}K$ good boxes and therefore the existence of an evenly spread component inside $Q_K$, the proof is complete.

\end{proof}

\begin{mydef}
	We call the evenly spread subgraph constructed in \Cref{prop:large_component} the \emph{distinguished subgraph} of $\cG$ inside $Q_K$.
	We call the $b\cdot K$ vertices used in the construction of the evenly spread subgraph from \Cref{prop:large_component} the \emph{distinguished} vertices of the distinguished subgraph.
\end{mydef}

\begin{remark}
	Given a distinguished subgraph and a vertex with mark from $\X$, %and its distinguished vertices
	the event that this vertex belongs to the same connected component inside $Q_K$ as the distinguished subgraph is an increasing event. Note that this does not hold if the distinguished subgraph is not given, nor does it necessarily hold for the event ``the vertex belongs to an evenly spread subgraph'', since the monotonicity of the event might fail depending on how the subgraph is chosen.
	 \end{remark}

%\todo[inline]{Sprinkling must begin here. Use $\epsilon$ process to connect get lower bound on probability that arbitrary vertex connects to component. Then use this to also get that \Cref{lemma:shared_vertex}}

For the following arguments, we will have to use sprinkling. To that end, we split the particle system $(\X_t)_{t\geq 0}$ into two parts, $(\X^{1-\varepsilon}_t)_{t\geq 0}$ and $(\X^{\varepsilon}_t)_{t\geq 0}$ for some arbitrary $\varepsilon\in(0,1)$. We do this by using the thinning property of Poisson point processes to thin $\X_0$ into the Poisson point processes $\X^{1-\varepsilon}_0$ of intensity $(1-\varepsilon)\lambda$ and $\X^{\varepsilon}_0$ of intensity $\varepsilon\lambda$ and let the two evolve (independently) over time using the same rules as before. We outline several useful facts about the resulting particle systems and graphs, most of them direct consequences of properties of Poisson point processes.

\begin{enumerate}[{Fact }1:]
	\item The systems $(\X^{1-\varepsilon}_t)_{t\geq 0}$ and $(\X^{\varepsilon}_t)_{t\geq 0}$ are stationary and independent of each other.\label{item:stationary}
	\item The graph $\cG^{\varepsilon}_t$ induced by $\X^{\varepsilon}_t$ is independent of the graph $\cG^{1-\varepsilon}_t$ induced by $\X^{1-\varepsilon}_t$.\label{item:graph_independent}
	\item A cube $Q_K$ that is $t$-$\alpha$-dense for $\X_s$ is $t$-$\varepsilon\alpha$-dense for $\X_s^{\varepsilon}$ and $t$-$(1-\varepsilon)\alpha$-dense for $\X_s^{1-\varepsilon}$.\label{item:dense}
	\item If $s,t\geq 0$ are such that $\lfloor s\rfloor\neq\lfloor t\rfloor$ and $\x,\y\in \X^{\epsilon}$, then conditionally on $\x_s,\x_t,\y_s,\y_t$, the events $\{\x_s\sim \y_s\}$ and $\{\x_t\sim \y_t\}$ are independent. The equivalent statement holds for vertices from $\X^{1-\epsilon}_t$ as well.\label{item:conditionallyIndep}
\end{enumerate}

\begin{lemma}\label{lemma:belonging}
	Let $\lambda>0$ and assume that $Q_K$ is $K$-$\alpha$-dense for some $\alpha>0$ and let $\gamma>\frac{\delta}{\delta+1}$. Let furthermore $G$ be the distinguished subgraph of $\cG^{1-\varepsilon}$ inside $Q_K$. Then, a given vertex $\x$ with a mark in $(0,1)$ and arbitrary location within $Q_K$ belongs to the same connected component of $\cG$ as $G$ with probability bounded uniformly in $K$ away from $0$, if $K$ is large enough.
	\end{lemma}
	
	\begin{proof}
	For the duration of the proof, we consider only vertices located inside $Q_K$ even when we do not explicitly say this. We will argue the claim in 3 steps:
	\begin{enumerate}[{Step }1:]
		\item For large enough $K$, $\x$ has a probability bounded away from $0$ of connecting to a vertex with mark smaller than $1/\log K$ belonging to $\X^{\varepsilon}$.
		\item For large enough $K$, any vertex with mark $s<1/\log K$ (and in particular the vertex from Step 1) of $\X^{\varepsilon}$ starts, with probability bounded away from $0$, a path of vertices from $\X^{\varepsilon}$ with ever decreasing marks (alternating through connectors), the final vertex having mark smaller than $K^{-\epst\theta \log 2}$.
		\item The final vertex of the path from Step 2 forms, for $K$ large enough, with probability bounded away from $0$ a connection via a connector from $\X^{\varepsilon}$ to the closest distinguished vertex in the bottom layer of $G$.
	\end{enumerate}
	Note that at every step, save for $\x$ and the distinguished vertex belonging to the bottom layer of $G$, every vertex considered belongs to $\X^{\varepsilon}$ and as such its location and mark are independent of $G$.
	
	Step 1 is an immediate consequence of \Cref{ass:main}. The vertices with mark smaller than $1/\log K$ that form a connection to $\x=(x,t)$ stochastically dominate a Poisson point process with intensity $\varepsilon\lambda\alpha(1\wedge\kappa_1 s^{-\delta\gamma}|x-y|^{-\delta d})$ on $Q_K\times(0,1/\log K)$. Consequently, the number of neighbours of $\x$ in $\X^{\varepsilon}\cap Q_K$ stochastically dominates a Poisson random variable with intensity
	\[
		\int_{Q_K}\mathd y\int_{0}^{1/\log K}\mathd s \varepsilon\lambda\alpha(1\wedge\kappa_1 s^{-\delta\gamma}|x-y|^{-\delta d}).
	\]
	Since the above integral can easily be checked to be increasing in $K$ and therefore bounded away from $0$ for large values of $K$, Step 1 follows by noting that a Poisson random variable with intensity bounded away from $0$ has probability of being equal $0$ bounded away from $1$. We denote the vertex with the smallest mark of these vertices by $\y$.
	
	In order to prove Step 2, we use a similar construction as in the proof of \Cref{prop:large_component}. To that end, let $y\in Q_K$ be the location of $\y$. Next, similar to the proof of \Cref{prop:large_component}, let $n_p=\lfloor K^{(1-\epst \log 2)/d}\rfloor$ and $k_p=\lfloor (\epst\log K)/d\rfloor$ and define
	\[
		B_{k}:=A_{k}\times\big(\tfrac{1}{2}e^{-(k+1)\theta d},\tfrac{1}{2}e^{-k\theta d}\big)\;\text{for }k=0,\dots k_p,
	\]
	with 
	\[
		A_{k}:=\bigtimes_{i=1}^d\big(2^k v_{i,k},2^k(v_{i,k}+1)\big),
	\]
	where $\v_k=(v_{1,k},\dots,v_{d,k})$ are chosen such that $y\in A_k$. Note that by construction, $A_{k}\subset A_{k+1}$ for every $k\in \{0,\dots,k_p-1\}$. As in the proof of \Cref{prop:large_component}, we also colour all of the vertices of $\X^{\varepsilon}$ with marks greater than $1/2$. Let $\epsilon_1>0$ small enough that $\theta>\frac{\epsilon_1+\log 2}{\gamma+\gamma/\delta}$.
	Choose $\epsilon_2\in(0,\theta\gamma\wedge \delta\epsilon_1)$ such that $\sum_{k=0}^\infty e^{-kd(\theta\gamma\wedge\delta\epsilon_1-\epsilon_2)}$ converges and colour the vertices of $\X^{\varepsilon}$ on $\R^d\times[\frac{1}{2},1)$ by the colour set $\N$ independently and such that the vertices with colour $k\in\N$ form a Poisson point process $\X^k$ on $\R^d\times[\frac{1}{2},1)$ with intensity proportional to $e^{-kd(\theta\gamma\wedge\delta\epsilon_1-\epsilon_2)}$. For $k=0,\dots,\k_p$, we denote by
	\[
		C_{k}:=\X^k\cap A_{k}\times[\frac{1}{2},1)
	\]	
	the \emph{potential connectors} (if $B_{k}$ is not empty). Then, using the same computation as in the proof of  \Cref{prop:large_component} and setting $\epsilon_3=\log 2-\theta>0$ as there, the probability that $B_k$ is not empty is greater than
	\[
		1-\exp\big\{(1-\alpha)c\varepsilon\lambda e^{kd\epsilon_3}\big\}.
	\]
	Likewise, given that $B_k$ and $B_{k+1}$ are not empty, the probability that there exists a connector vertex from $C_k$ that is connected to the oldest vertex in $B_k$ and the oldest vertex in $B_{k+1}$ is greater than
	\[
		1-\exp\{-c\varepsilon\lambda e^{kd\epsilon_2}\}.
	\]
	Note that this bound also holds if instead of the oldest vertex inside $B_k$ one uses an arbitrary given vertex from $B_k$, which we implicitly use next.
	
	Let now $s$ be the mark of $\y$ and $k^*$ chosen such that $\y\in B_{k^*}$. Then, using the above bounds, the probability that there exists a path of vertices alternating between $B_k$ and $C_k$  as $k$ goes from $k=k^*$ to $k_p-1$, starting from $\y$ is greater than
	\[
		1-\sum_{k=k^*}^{k_p-1}\exp\{-c\varepsilon\lambda e^{kd(\epsilon_2\wedge\epsilon_3)}\}>1-\sum_{k=k^*}^{\infty}\exp\{-c \varepsilon\lambda e^{kd(\epsilon_2\wedge\epsilon_3)}\}.
	\]
	Since the infinite sum above is summable and $k^*$ is increasing in $K$, this probability is bounded can be made strictly positive for large enough $K$. We denote by $\z$ the vertex of this path that lies in $B_{k_p}$, i.e. the final vertex of the path.
	
	For Step 3, we begin by noting that the bottom vertices of $G$ have marks smaller than $K^{-\epst\theta d}$. Similarly, $\z$ has a mark smaller than $K^{-\epst\theta d}$. By construction of $G$, there exists at least one vertex among the bottom vertices that is at distance at most $2^{(\epst\log K)}$ from $\z$, and by the same argument as in the proof of \Cref{prop:large_component}, the probability that the two vertices are connected via a connector from $C_{k_p+1}$ grows exponentially close to $1$ as $K$ is made large and is therefore bounded away from $0$ for $K$ large enough.
	
	Combined, this gives the desired claim.
	\end{proof}

\begin{lemma}\label{lemma:shared_vertex}
	Let $G_i$, $i\in\{1,2\}$ be the distinguished subgraphs of $\cG^{1-\varepsilon}$ inside $Q_K$ at times $t_1$ and $t_2$ such that $\lfloor t_1\rfloor\neq\lfloor t_2\rfloor$. Then, the probability that there exists at least one vertex $\x\in\X^{\varepsilon}$ inside $Q_K$ that belongs to the same connected component of $\cG$ as $G_i$ at time $t_i$, for $i\in\{1,2\}$, is larger than
	\[
		1-\exp\{-c(1-\exp\{-\tfrac{K^{2/d}}{t_2-t_1}\})\varepsilon\lambda K\},
	\]
	for some constant $c>0$.
\end{lemma}

\begin{proof}
	%\todo[inline]{Sprinkling: see proof in Peres, Sinclair, Sousi, Stauffer}
	We first consider the probability that a vertex that is located inside the middle (say at distance at least $K^{1/d}/4$ from the boundary) of $Q_K$ at $t_1$ is located inside $Q_K$ at time $t_2$. Using the bound on the probability that a vertex moves further than $(K^{1/d}/4)$ in $t_2-t_1$ time, this probability is larger than
	\[
		1-\exp\left\{-c\tfrac{K^{2/d}}{t_2-t_1}\right\},
	\]
	for some positive constant $c$.
		
	We can, without loss of generality, sample $\x\in\X^{\varepsilon}$ independently of any other vertices from $\X^{\varepsilon}$ used in the rest of the arguments, by simply splitting $\X^{\varepsilon}$ further into two disjoin Poisson point processes, each with intensity $\frac{\varepsilon}{2}$ and sampling $\x$ from the first of the two, with the rest of the vertices used in the other sprinkling arguments being sampled from the second. We omit this detail in the rest of the argument to keep notation concise.
	
	By \Cref{lemma:belonging}, any vertex of $\X$ inside $Q_K$ belongs to the same connected component as $G_1$ (resp. $G_2$) with probability at least $p>0$, with $p$ bounded away from $0$. By the independence of the point processes $\X^{1-\varepsilon}_{t_1}$, $\X^{1-\varepsilon}_{t_2}$, $\X^{\varepsilon}_{t_1}$ and $\X^{\varepsilon}_{t_2}$, a given vertex $\x\in\X^{1-\varepsilon}$ belongs to the connected component of $G_{t_1}$ at time $t_1$ and $G_{t_2}$ at time $t_2$ with probability at least $p^2$. Therefore, using the thinning property of Poisson point processes, the vertices of $\X_0^{\varepsilon}$ that are located in the middle of $Q_K$ at $t_1$  and are inside $Q_K$ at $t_2$, that belong to both connected components at their respective times, stochastically dominate a Poisson point process in $Q_K$ of intensity $(1-\exp\{-\frac{K^{2/d}}{t_2-t_1}\})p^2\varepsilon\lambda$. The probability of at least one such vertex existing is therefore at least $1-\exp\{-(1-\exp\{-\frac{K^{2/d}}{t_2-t_1}\})p^2\varepsilon\lambda K\}$.
\end{proof}

%\begin{corollary}\label{corol:shared_vertex}
%	Let $G_1$ be the distinguished subgraph of $\cG^{1-\varepsilon}$ inside the cube $Q_{K}^1$ and let $G_2$ be the distinguished subgraph of $\cG^{1-\varepsilon}$ inside the cube $Q_{K}^2$. Assume that the volume of the intersection $Q_{K}^1\cap Q_{K}^2$ is at least $c_1 K^d$ for some constant $c_1>0$. Then the probability that $G_1$ and $G_2$ belong to the same connected component of $\cG$ is at least 
%	\[
%		1-\exp\{-c_2\varepsilon K^d\},
%	\]
%	for some constant $c_2>0$.
%\end{corollary}
%
%\begin{proof}
%	Like in the proof of \Cref{lemma:shared_vertex}, we use \Cref{lemma:belonging} to bound the probability of an arbitrary vertex that lies in $Q_{K}^1\cap Q_{K}^2$ belonging to the same connected component as $G_1$ (resp. $G_2$). In order to avoid correlations between the two events, we can use the familiar sprinkling argument, by first splitting the Poisson point process $\X^{\varepsilon}$ in two, and then use the first half of the process to bound the probability of the first event and the second for the second event. Consequently, the probability that a vertex exists in $Q_{K}^1\cap Q_{K}^2$ that belongs to the connected component of $G_1$ and $G_2$ is at least
%	\[
%		1-\exp\{- p^2\frac{1}{4} \varepsilon c_1K^d\},
%	\]
%	with $p>0$ bounded away from $0$ by \Cref{lemma:belonging}.
%\end{proof}

\section{Proofs of theorems}\label{sec:proofs}
\subsection{Broadcast time}
Recall that we consider broadcast time on $\mathbb{T}_n^d$, the $d$-dimensional torus of volume $n$.

%Before proceeding with the proof of the theorem, we first explain how the sequence of random graphs $(\cG_t)_{t\geq 0}$ is constructed on the torus. The particles remain distributed like a Poisson point process of unit intensity inside $Q_n\times(0,1)$, while the intensity is zero on the rest of the marked vertex space $Q_n^c\times(0,1)$. The vertices still move according to Brownian motion on $\R^d$, but we let their movement ``wrap around'' the torus when they reach the boundary by projecting their location onto $Q_n$.

\begin{proof}[Proof of \Cref{thm:flooding}]
	Let $t=C\log n(\log\log n)^\varepsilon$ for some sufficiently large constant $C$ and any $\varepsilon>0$. Let $\epsilon_1>0$ be small but fixed. We will consider the broadcast of information over times $i\in[0,t]$. Note that if we want to apply \Cref{prop:alpha_dense} with this choice of $t$, the cube $Q_L$ will have side-length $t\ll n$. We will therefore apply \Cref{prop:alpha_dense} across a tessellation of the torus into cubes of side length $t$. For each such cube $Q_L$, at the $(1-\epsilon_1)t$ times when the cube $Q_L$ is dense, using \Cref{prop:large_component}, the cube $Q_L$ contains a distinguished subgraph with probability at least $1-\exp\{-L^{\epsilon_2}\}$. 
	
	Taking a union bound across the $n/t^d$ many cubes in the torus, the probability that each of the $n/t^d$ cubes contains a distinguished subgraph at all of the $(1-\epsilon_1)t$ times when the cube is $t$-$\epsilon_1$-dense is greater than
	\[
		1-\frac{n}{t^d}\exp\{-c_1t^{c_2}\}-t(1-\epsilon_1)\frac{n}{t^d}\exp\{-t^{\epsilon_2}\},
	\]
	where $c_1$ and $c_2$ are the constants from \Cref{prop:alpha_dense}.
	Furthermore, by \Cref{lemma:shared_vertex}, the probability that inside the given cube $Q_L$, all of the distinguished components always share at least one vertex with their preceding distinguished component is at least
	\[
		1-t\exp\{-c(1-\exp\{-\tfrac{t^{2}}{\epsilon_1 t}\})t\}
	\]
	where we have omitted the sprinkling parameter from the lemma, since it is arbitrary and constant, in order to not confuse notation. Note that the term $\epsilon_2t$ in the expression represents the worst case time difference between two subsequent observations of the distinguished component inside of a given cube.

	Next, from \cite[Equation (3.1)]{Jorritsma2023a} we have that the probability the second largest connected component of the torus of volume $n$ is larger than $k$ is
	\[
		\P(|C_n^{(2)}|\geq k)\leq \frac{n}{c}\exp\{-ck^{\zeta_{hh}},\}
	\]
	where $\zeta_{hh}$ depends only on the model parameters.
	Consequently, for our choice of $t$, the probability that any of the distinguished components at time $i$ does not belong to the largest connected component on the torus is smaller than
	\[
		\frac{n}{c}\exp\{-c (bC\log n(\log\log n)^\varepsilon)^{\zeta_{hh}/d}\}=\frac{n}{c}n^{-c (bC)^{\zeta_{hh}/d}(\log\log n)^{\varepsilon\zeta_{hh}/d}},
	\]
	which for $C$ chosen large enough can be made decreasing (and even summable) in $n$. We therefore have that the largest connected component on the torus contains in every cube of the tessellation, for at least $(1-\epsilon_2)t$ times, a distinguished subgraph of size at least $b\cdot t^d$ with probability at least 
	\begin{equation}
		\begin{split}
		&1-\frac{n}{t}\exp\{-c_1t^{c_2}\}-t(1-\epsilon_1)\frac{n}{t}\exp\{-t^{\epsilon_2}\}\\
		&\quad-\frac{n}{t}t\exp\{-c(1-\exp\{-\tfrac{t^{2}}{\epsilon_1 t}\}) t\}-t\frac{n}{c}n^{-c (bC)^{\zeta_{hh}/d}(\log\log n)^{\varepsilon\zeta_{hh}/d}},
		\end{split}\label{eq:denseTorus}
	\end{equation}
	which tends to $1$ as $n$ grows large.
	
	We now finally consider the propagation of information. First, the probability that the origin vertex does not belong to the largest connected component at any time during $[0,t/2)$ is smaller than the probability that it does not belong to the same connected component as the distinguished subgraph inside the cube $Q_L$ to which the origin belongs, for all of the at least $(\frac{1}{2}-\epsilon)t$ times in $[0,t/2)$ when the distinguished subgraph exists. Let $p$ be the probability from \Cref{lemma:belonging} that a vertex with an arbitrary mark belongs to the same connected component as the distinguished subgraph. Then, due to the above, the probability that the origin does not belong to the largest connected component during $[0,t/2)$ is smaller than $(1-p)^{(\frac{1}{2}-\epsilon)t}$.
	Consequently, we have with probability
	\begin{equation}
		1-(1-p)^{(\frac{1}{2}-\epsilon)t}\label{eq:originBroadcast}
	\end{equation}
	that the largest connected component on the torus has already received the broadcast by time $\frac{t}{2}$. Each vertex of the graph therefore has $\frac{t}{2}$ time to receive the information by connecting to the largest connected component. Using a simple Chernoff bound, the number of vertices on the torus is at most $(1+\delta)n$ with probability larger than $1-e^{-\Omega(n)}$ and on this event, the probability that each of the vertices belongs to the largest connected component at least once during $[t/2,t]$ is at least
	\[
		1-(1+\delta)n(1-p)^{(\frac{1}{2}-\epsilon)t}.
	\]
	Combining this with \eqref{eq:denseTorus} and \eqref{eq:originBroadcast}, recalling that $t=C\log n(\log\log n)^\varepsilon$ and setting $C$ large enough then gives the result.	
\end{proof}

\subsection{Percolation time}

\begin{proof}[Proof of \Cref{thm:percolation}]
	For $j\in\N_0$ define
	\[
		A_j:=\{\0_j\text{ does not belong to }\mathcal{C}_\infty^j\},
	\]
	that is the event that at time $j$, the origin vertex $\0$ does not belong to the infinite connected component $\mathcal{C}_\infty^j$. It then follows, that
	\[
		\P(T_{\text{perc}}>t)\leq\P\left(\cap_{j=0}^{\lfloor t\rfloor} A_j\right)=:\P(B_t).
	\]
	
	Consider now the cube $Q_{L^d}$ for $L=t$. On the event that at time $j$ there exists a distinguished subgraph inside of $Q_{L^d}$, let $\tilde A_j$ to be the event that the origin vertex $\0$ does not belong to the same connected component as the distinguished subgraph at time $j$. In analogy to $B_t$, define $\tilde B_t$ to be the event that $\0$ does not belong to the same connected component as the distinguished subgraph for any of the discrete time steps $j=0,1,\dots,\lfloor t\rfloor$. Next, let $C_j$ be the event that at time $j$, the distinguished subgraph belongs to the infinite connected component $\mathcal{C}_\infty^j$ and let $D_t:=\cap_{j=0}^{\lfloor t\rfloor}C_j$ be the event that at all discrete times $j=0,1,\dots,\lfloor t\rfloor$, the distinguished subgraph inside $Q_{L^d}$ is a subgraph of the infinite component.
	
	Let $\epsilon_t>0$ be a small constant and let $E_t$ be the event that for at least $(1-\epsilon_t)t$ time steps, the cube $Q_{L^d}$ is $L^d$-$\xi$-dense, where $\xi\in(0,1)$ can be chosen arbitrarily, but should be thought of as small. Finally, let $F_t$ be the event that an evenly spread component exists at each of the at least $(1-\epsilon_t)t$ times when the cube $Q_{L^d}$ is dense.
	
	We then first note that
	\begin{equation}
		\begin{split}
			\P(B_t)&\leq \P(B_t\cap F_t\cap E_t) + \P(F_t^c\cap E_t)+ \P(E_t^c)\\
			&\leq\P(B_t\cap F_t\cap E_t) + t\exp\{-L^{d\epsilon}\} + \exp\{-c_1 t^{c_2}\},
		\end{split}\label{eq:split1}
	\end{equation}
	where we used \Cref{prop:large_component,prop:alpha_dense} and a union bound across the at most $\lfloor t\rfloor$ time steps in the second inequality.
	
	Next, if $\tilde A_j^c$ holds for some $j\in\{0,\dots,\lfloor t\rfloor\}$ and $D_t$ holds, then $B_t$ cannot hold and consequently
	\begin{equation}
		\P(B_t\cap F_t\cap E_t)\leq \P((\tilde B_t\cup D_t^c)\cap F_t\cap E_t)\leq \P(\tilde B_t\cap F_t\cap E_t)+\P(D_t^c\cap F_t\cap E_t).
		\label{eq:split2}
	\end{equation}
	We can bound $\P(D_t^c\cap F_t\cap E_t)$ using the results of \cite[Theorem 2.2]{Jorritsma2023a} by
	\begin{equation}
		\P(D_t^c\cap F_t\cap E_t)\leq t c\exp\left\{-c^{-1}(b L^{d})^{\zeta_{hh}}\right\},
		\label{eq:split3}
	\end{equation}
	where $c$ and $\zeta_{hh}$ are two positive constants that depend on the model parameters only, and the $t$ term comes from taking a union bound across the $\lfloor t\rfloor$ time steps. Note that by \cite{Gracar2021} the above bound holds for any $\lambda>0$.
	
	We next proceed to bound $\P(\tilde B_t\cap F_t\cap E_t)$. By \Cref{lemma:belonging}, 
	\begin{equation}
		\P(\tilde B_t\cap F_t\cap E_t)\leq (1-p)^{(1-\epsilon_t)t},
		\label{eq:split4}
	\end{equation}
	where $p$ is the (uniformly in $L^d$) positive probability that an arbitrary vertex belongs to the same connected component as the distinguished subgraph. Putting \eqref{eq:split1}, \eqref{eq:split2}, \eqref{eq:split3} and \eqref{eq:split4} together, we obtain
	\[
		\P(B_t)\leq \exp\{-ct^{\frac{1}{c}}\}
	\]
	for some sufficiently large constant $c$, where we have used that $L=t$, which completes the proof.
	\end{proof}

\appendix
\section{Known results}

\begin{lemma}\label{lemma:twoconn}
Given two vertices $\x=(x,u)$ and $\y=(y,v)$ with marks smaller than $\frac{1}{2}$, the number of vertices with mark larger than $\frac{1}{2}$ which form an edge to both $\x$ and $\y$ is Poisson-distributed with parameter larger than 
\begin{equation}
C u^{-\gamma} \big(1\wedge v^{-\gamma\delta}\big(\abs{y-x} + u^{-\gamma/d}\big)^{-d\delta}\big), \label{eq:twoconn}
\end{equation}
where $C>0$ is a constant that does not depend on $k$.
\end{lemma}
\begin{proof}
We consider the vertices $\z = (z,w)$ with $\abs{x-z}^d < u^{-\gamma}$. Note that these vertices lie close enough to $\x$ that they form an edge to it with probability at least $\alpha(1\wedge\kappa_1)$. Then, the number of those vertices which form an edge to $\x$ and $\y$ is Poisson-distributed with parameter bounded from below by
\begin{align*}
\int_{\frac{1}{2}}^{1} \int_{B(x,u^{-\gamma/d})} \mathd \z &\alpha^2(1\wedge \kappa_1)(1\wedge \kappa_1 v^{-\gamma\delta} \abs{z-y}^{-d\delta})\\
\geq &\frac{V_d \alpha^2(1\wedge \kappa_1)}{2}u^{-\gamma}\big(1\wedge \kappa_1 v^{-\gamma\delta}\big(\abs{y-x} + u^{-\gamma/d}\big)^{-d\delta}\big),
\end{align*}
where $V_d$ is the volume of the $d$-dimensional unit ball and $\alpha$ comes from Assumption \ref{ass:main}. Thus, there exist a constant $C>0$ sufficiently small such that \eqref{eq:twoconn} holds.
\end{proof}

\begin{lemma}[Chernoff bound for Poisson]
\label{lemma:pchernoff}
Let \(P\) be a Poisson random variable with mean \(\lambda\). Then, for any \(0<\epsilon<1\),
\[
	\mathbb{P}(P<(1-\epsilon)\lambda) < \exp\{-\lambda\epsilon^2/2\}
\]
and
\[
	\mathbb{P}(P > (1 + \epsilon)\lambda) < \exp\{-\lambda\epsilon^2/4\}.
\]
\end{lemma}

\begin{lemma}[{Large deviation for Binomial, \cite[Corollary A.1.10]{Alon2000}}]
\label{lemma:bchernoff}	
Let $B$ be a binomial random variable with parameters $n$ and $p$. Then,
\[
	\P(B\geq np+a)\leq\exp\left\{a-(pn+a)\log\left(1+\tfrac{a}{pn}\right)\right\}
\]
\end{lemma}

{\bf Acknowledgment:} This research was supported by the \emph{Deutsche Forschungsgemeinschaft} (DFG) as Project Number~425842117.

\printbibliography
\end{document}